\documentclass[twoside,a4paper,reqno,11pt]{amsart}
\usepackage{amscd,amsmath,amssymb,amsfonts}
\usepackage{url,enumerate,hyperref}
\usepackage[pdftex]{color,graphicx}
\usepackage{dsfont}
\usepackage[cmtip, all]{xy}

\usepackage[top=30mm,right=30mm,bottom=30mm,left=20mm]{geometry}

\newtheorem{propo}{Proposition}[section]

\newtheorem{lemma}[propo]{Lemma}
\newtheorem{corol}[propo]{Corollary}

\newtheorem{theo}[propo]{Theorem}

\newtheorem{remar}[propo]{Remark}

\numberwithin{equation}{section}

\newcommand{\mar}{\marginpar}

\newcommand{\bl}{\begin{lemma}\label}
\newcommand{\el}{\end{lemma}}

\newcommand{\ld}{,\ldots ,}

\newcommand{\ra}{ \rightarrow }

\newcommand{\lan}{ \langle }
\newcommand{\ran}{ \rangle }

\newcommand{\diag}{\mathop{\rm diag}\nolimits}

\newcommand{\Id}{\mathop{\rm Id}\nolimits}
\newcommand{\Irr}{\mathop{\rm Irr}\nolimits}

\newcommand{\spec}{\mathop{\rm Spec}\nolimits}

\newcommand{\CC}{{\mathbb C}}
\newcommand{\FF}{{\mathbb F}}

\newcommand{\QQ}{{\mathbb Q}}

\newcommand{\ZZ}{{\mathbb Z}}

\newcommand{\al}{\alpha}
\newcommand{\be}{\beta}

\newcommand{\ep}{\varepsilon}
\newcommand{\lam}{\lambda }

\newcommand{\om}{\omega }

\newcommand{\StC}{\mathsf{St}}
\newcommand{\Fr}{\mathsf{F}}
\newcommand{\GC}{\mathbf{G}}

\newcommand{\up}{^{-1}}

\def\d12{{_{12}}}

\def\au{{automorphism }}

\def\ei{{eigenvalue }}
\def\eis{{eigenvalues }}

\def\f{{following }}

\newcommand{\med}{\medskip}

\def\ii{{if and only if }}

\def\ir{{irreducible }}

\def\irr{{irreducible representation }}

\def\st{Suppose that }

\def\itf{{It follows that }}

\def\mult{{multiplicity }}

\def\po{{polynomial }}
\def\pos{{polynomials }}

\def\rep{{representation }}
\def\reps{{representations }}

\def\syl{{Sylow $p$-subgroup }}

\def\PSL{{\rm PSL}}
\def\SL{{\rm SL}}
\def\SO{{\rm SO}}
\def\GO{{\rm O}}
\def\Spin{{\rm Spin}}
\def\SU{{\rm SU}}
\def\GL{{\rm GL}}
\def\Sp{{\rm Sp}}

\newcommand{\EC}{\mathcal{E}}

\newcommand{\reg}{\mathrm{reg}}

\newcommand{\bp}{\begin{proof} }
\newcommand{\enp}{\end{proof}}

\newcommand{\edit}[1]{{\color{black} #1}}
\newcommand{\editnew}[1]{{\color{black} #1}}

\begin{document}

\title[Hall-Higman type theorems for exceptional  groups, I]{Hall-Higman type theorems for exceptional  groups of Lie type, I}

\author{Pham Huu Tiep}
\address{Department of Mathematics,
Rutgers University, Piscataway, NJ 08854, U.S.A.}
\email{tiep@math.rutgers.edu}
\author[A. E. Zalesski]{A. E. Zalesski}
\address{Department of Physics, Mathematics and Informatics, National Academy of Sciences of Belarus, 66 Prospekt  Nezavisimosti, Minsk, Belarus}
\email{a.zalesskii@uea.ac.uk}
\thanks{The first author gratefully acknowledges the support of the NSF
(grant DMS-1840702), the Joshua Barlaz Chair in Mathematics, and the Charles Simonyi Endowment at the Institute for
Advanced Study (Princeton).}
\date{}
\maketitle

\begin{abstract}
The paper studies the minimum  \po degrees of $p$-elements in cross-characteristic \reps  of simple groups of exceptional Lie type whose BN-pair rank is at most 2. Specifically, we prove that the degree in question equals
the order of the element.
\end{abstract}

\maketitle

\section{Introduction}

This paper continues our earlier work \cite{TZ8} devoted to generalize the famous Hall-Higman  theorem on the minimum \pos of $p$ elements in \reps of $p$-solvable groups to more general classes of groups.
The bulk of the project is the case of almost simple groups. The paper \cite{TZ8} deals mainly with classical groups, and this paper completes the project for quasi-simple groups of BN-pair rank at most 2.
Specifically, we prove the following result.

\begin{theo}\label{mth1}
Let $G\in\{{}^2B_2(q), q>2$, ${}^2G_2(q),q>3$, ${}^2F_4(q)$, $G_2(q)$, ${}^3D_4(q)\}$.  Let  $g\in G$ a element of prime power order coprime to $q$. Let $\phi$ be a non-trivial \irr of $G$ 
over a field F of characteristic $\ell$ coprime to q. Then the minimum \po degree of $\phi(g)$ equals $|g|$,
unless possibly when $G={}^2F_4(8)$, $\ell=3$, $p=109$ and   $\phi(1)<64692$.
\end{theo}

\editnew{Observe that that it suffices to prove Theorem \ref{mth1} for $F$ algebraically closed.}
Theorem \ref{mth1} is valid for the Tits group which is a subgroup of index 2 in  ${}^2F_4(2)$; also see Lemma \ref{2g24} for $G=2\cdot G_2(4)$.

Theorem \ref{mth1} improves our earlier result \cite[Theorem 4.6]{TZ8}, stating that $\deg\phi(g)\geq|g|(1-1/p)$ whenever a \syl of $G$ is cyclic.

In some special cases the result of Theorem \ref{mth1} was known earlier. These are

\begin{enumerate}[\rm(i)]
\item Sylow $p$-subgroups of the quasi-simple group $G$ are cyclic and   $\ell\in\{0,p\}$ \cite{Z2};

\item $\ell=0$, $p>2$ and $G\in\{G_2,{}^2F_4(q),{}^2F_4(2)',{}^3D_4(q) \}$  \cite[Lemmas 4.11 and 4.14]{Z4};

\item $G\cong G_2(q)$, $p>2$ and $g$ lies in a parabolic subgroup of $G $ \cite[Lemma 4.10]{Z4}.
\end{enumerate}


\med
{\bf Notation.} Let $G$ be a finite group. Then $|G|$ is the order of $G$, $Z(G)$ be the center of $G$ and $O_p(G)$ the maximal normal $p$-subgroup of $G$ for a prime $p$. We often use $|G|_p$ to denote the $p$-part of $|G|$. For $g\in G$ the order of $g$ is denoted by $|g|$ \edit{and $o(g)$ is the order of $g$ modulo $Z(G)$}. A $p'$-element is one of order coprime to $p$.

\editnew{$\FF_q$ means the finite field of $q$ elements, $\overline{\FF}_q$ its algebraic closure and $\QQ,\CC$ are the rational and complex number fields, respectively. $\ZZ$ denotes the set of integers.}

Let $F$ be  an algebraically closed field of characteristic $\ell$,   and $\phi$  an $F$-\rep of $G$. Then $\deg\phi(g)$ denotes the minimum \po degree of $\phi(g)$.  We write
$\phi\in\Irr_\ell G$ to indicate that $\phi$ is irreducible, and use this notation for  the Brauer character of $\phi$ too.
 If $\ell=0$, we drop the subscript $\ell$.
If $\chi$ is an ordinary (generalized) character of $G$, and with $\ell$ a fixed prime, then  $\chi^\circ$  is the restriction of $\chi$ to $\ell'$-elements.

We denote by $1_G$ the trivial character of $G$ (both ordinary and $\ell$-modular), and by $\rho_G^{\reg}$ the regular \rep of $G$ or the (Brauer) character of it. The ordinary Steinberg \rep and its character of a group of Lie type is denoted by $\StC_G$ or $\StC$.

 A Brauer character  $\phi\in\Irr_\ell G$ is called {\it liftable} if there exists an ordinary character $\tau$ of $ G$ such that $\tau^\circ=\phi$. An \irr or $FG$-module is called liftable if the Brauer character of it is liftable.

If $H$ is a subgroup of $G$ and $\eta$ is a character or \rep of $H$ then $\tau^G$ denotes the induced character. If $\phi$ is a character or \rep of $G$ then $\phi|_H$ stands for the restriction of $\phi$ to $H$.

If $V$ is an $FG$-module and $X$ a subset of $G$ we write $V^X$ or $C_V(X)$ for the subspaces of elements
of $V$ fixed by all $x\in X$.

For integers $a,b>0$ we write $(a,b)$ for the greatest common divisor of $a,b$ and $a|b$ means that $b$ is a  multiple of $a$. We also write $(\chi,\phi)$
for the inner product of characters $\chi,\phi$ of a group $G$. \editnew{We write $\diag(x_1\ld x_m)$ for a block-diagonal matrix with diagonal blocks $x_1\ld x_m$. }

A finite group of Lie type is that of shape $\mathbf{G}^\Fr$, where $\mathbf{G}$ is a reductive algebraic group and $\Fr$ a Frobenius endomorphism of it. For more information see \cite{C} or \cite{DM}.


\section{Preliminaries}

\begin{lemma}\label{tr1}
 Let $\,C$ be a cyclic group of order coprime to $\ell$, and $\chi$ a non-trivial Brauer character of $C$ such that $\chi(c)=a$ for every
$1\neq c\in C$ (so $\chi$ is constant on $C\smallsetminus\{1\}$).
Then $\chi=\frac{\chi (1)-a}{|C|}\cdot \rho_C^{\reg} +a\cdot 1_C$. In particular, if $ a\geq 0$ or  $a<0$ and $-a\cdot (|C|-1) <\chi(1)$ then $\chi=\rho_C^{\reg}+\chi'$
for some proper character $\chi'$ of C.
\end{lemma}

\begin{proof} The first claim is obvious as $\chi-a\cdot 1_G$ vanishes at $C\smallsetminus \{1\}$. Let $a\geq 0$; then $\frac{\chi (1)-a}{|C|}>0$ as
this equals to the \mult of any non-trivial \ir character of $C$ in $\chi$.  Let $0<-a<|C|$. Then $\frac{\chi (1)-a}{|C|}\cdot \rho_C^{\reg}+a\cdot 1_C=\frac{\chi (1)+a( |C| -1)}{|C|}\cdot \rho_C^{\reg}-a\cdot \rho_C^{\reg} +a\cdot 1_C$, so $\chi (1)+a( |C|-1)>0$ implies the second claim in this case.
\end{proof}

\begin{lemma}\label{zgm} {\rm \cite[Lemma IX.2.7]{HB}}
Let $p,r$ be primes and $a,b$ positive integers such that $p^{a}=r^{b}+1$. Then either

\begin{enumerate}[\rm(i)]
\item $p = 2$, $b = 1$, and $r$ is a Mersenne prime, or

\item $r = 2, a = 1$, and $p$ is a Fermat prime, or

\item $p^{a} = 9$. $\hfill \Box$
\end{enumerate}
\end{lemma}

\bl{hh1}
{\rm \cite[Theorem 1.10]{HB}} Let  
$G=AH$ be a semidirect product, where $A$ is an abelian normal subgroup of $G$ and $H=\lan h\ran$
is a p-group. Let $\phi$ be a faithful $\ell$-\rep of G. Suppose that  $(p\ell,|A|)=1$ and $C_A(h)=Z(G)$. Then $\deg \phi(h)=o(h)$.\el

\begin{corol}\label{hh2}
Let $\SL_n(q)\leq G\leq  \GL_n(q)$ and let $h\in G$ be a non-central $p$-element. Let $\phi$ be an $\ell$-modular
 \rep of G such that ${\rm ker}\,\phi\leq Z(G)$ and $\ell \nmid q$. Suppose that $p \nmid q$ and $h$ is not \ir on the natural $\GL_n(q)$-module $\FF_q^n$.
Then $\deg \phi(h)=o(h)$. \end{corol}

 \bp Let $V=\FF_q^n$ and let $W\neq 0$ be a proper $h$-stable subspace of $V$. Let
 $$A:=\{g\in G \mid gW=W\mbox{ and }g\mbox{ acts trivially on both }W\mbox{ and }V/W\}.$$
 Then $A$ is an abelian group,  $(|A|,p\ell)=1$ and $hAh\up=A$. In addition, $C_A(h)\leq Z(G)$ (in fact $C_{\GL_n(q)}(A)=AZ(G)$). So the result follows from Lemma \ref{hh1}.\enp

\begin{lemma}\label{p4a}
Let $G=\Sp_{2n}(q)$, $q$  even, $n>1$, and let $g\in G$ be a reducible $p$-element for $p\not|q$. Let $\phi\in  \Irr_\ell G$   with $\dim\phi>1$ and $l \neq 2$.  Then $\deg\phi(g)=|g|$.\end{lemma}

\bp If $g$ belongs to a parabolic subgroup of $G$ then the result is contained in \cite{DZ1}. Otherwise, $|g|>3$ and  $g\in H\cong H_1\times H_2$, where $H_1\cong \Sp_{2k}(q)$, $H_2\cong \Sp_{2l}(q)$, $k+l=n$.
Then $g=g_1 g_2$, where $g_1\in H_1$, $g_2\in H_2$ are $p$-elements. We may assume that $|g|=|g_1|$
and moreover that $g_1$ is \ir in $H_1$. In addition, we may assume that $(k,q)\neq (1,2)$ as otherwise
$|g|=3$.

By \cite[Corollary 3.8]{TZ8},  the restriction $\phi|_H$ contains an \ir constituent $\tau=\tau_1\otimes \tau_2$ ($\tau_1\in \Irr_\ell H_1$, $ \tau_2\in \Irr_\ell H_2$) such that
 both $\dim\tau_1,  \dim\tau_2>1$. Then $\tau(g)=\tau_1(g_1)\otimes \tau_2(g_2)$.
If $\deg\tau_1(g_1)=|g_1|$ then we are done. Otherwise, by \cite[Lemma 3.3 and Prop. 5.7]{TZ8}, $|g_1|=q^k+1$ and $\deg\tau_1(g_1)\geq |g_1|-2$ (in fact, $\deg\tau_1(g_1)\geq |g_1|-1$ if $k>1$).
By Lemma \ref{zgm}, $|g_1|=q^k+1$ implies either  $|g_1|=p=q^k+1$ or $q^k=8$, $|g_1|=9$.

 Suppose first that $|g_1|=p>3$. Then   $\deg\tau_2(g_2)\geq 3$ by \cite[Theorem 1.2]{TZ8}, and hence
 $\deg\tau(g)=p$  by \cite[Lemma 2.12]{TZ8}.

 Suppose that $|g|=9,q=8$. Then $\deg\tau_2(g_2)\geq 3$ and $\deg\tau_1(g_1)\geq 7$.    If $\ell>3$ then \cite[Lemma 2.12(i)]{TZ8} again yields the result.  Let $\ell=3$. Let $J_i$ denote the Jordan block of size $i$ over $\FF_3$. Then the minimum \po degree of $J_7\otimes J_3$ equals 9 \cite[Lemma 2.11]{Su9});
 in addition, if $a\geq 7,b\geq 3$ then the minimum \po degree of $J_a\otimes J_b$ is at least 9
 \cite[Lemma 2.10]{Su9}. Therefore, $\deg\tau(g)=9$.

 Let $q=2,k=3$ and $|g|=9$. Then $\deg\tau_1(g_1)\geq 7$ by  \cite[Prop. 5.7(ii)]{TZ8}. If
 $\deg\tau_2(g_2)\geq 3$, then the result follows as above. Suppose that $\deg\tau_2(g_2)=2$.
 Then $|g_2|=3$. If $H_2\neq \Sp_2(2)$ then $g$ is contained in a parabolic subgroup of $G$.
 So $l=2$, $n=4$. In this case we show that $\phi|_{H_1}$ has an \ir constituent of degree
 greater than 7. Indeed, $H_1$ is contains in a parabolic subgroup $P$, the stabilizer of a line
 at the natural $\FF_2G$-module. Let $Q=O_2(P)$. Then $Q$ is an abelian group.
 Then $\phi|_Q$ is a direct sum of linear \reps $\lam$ of $Q$ permuted by $P$ when $P$ acts on $Q$ by conjugation. Then there is $\lam$ whose $\lan g_1\ran$-orbit is faithful.
 Let $\Lambda=H_1\lam$ be the $H_1$-orbit of $\lam$ with point stabilizer $C_{H_1}(\lam)\cong P_1$, where
 $P_1$ is the stabilizer of a vector at the natural $\FF_2H_1$-module.
 This yields a permutational $FH_1$-module $L\cong 1_{P_1}^{H_1}$. 
The composition factors of $L$ are the reduction modulo 3 of those for the similar module over
the complex numbers. The latter decomposes as $1_{H_1}+\chi_1+\chi_2$, where $\chi_1(1)=27$ and
   $\chi_2(1)=35$ \cite[p. 46]{Atlas}. As $\chi^\circ _1$, the restriction of $\chi$ to $3'$-elements, coincides with
   the Brauer character of $H_1$ of degree 27 \cite{JLPW}, we conclude that $L$ has an \ir constituent of degree 27. This contradicts \cite[Prop. 5.7(ii)]{TZ8}.\enp

\begin{lemma}\label{ss3}
Let   $G=\SL_3(q)$, $q>2$, and let $\phi\in\Irr_F(G)$,   $\dim\phi>1$, $\ell \nmid q$. Let $g\in G$
be a $p$-element. Then either $\deg\phi(g)=o(g)$, or $(3,q-1)=1$, $|g|=q^2+q+1$; moreover,  $\dim\phi=q^2+q-1$ if $p=\ell$, whereas  if  $p\neq \ell$  then $\deg\phi(g)=\dim\phi=q^2+q$ and  $1$ is not an eigenvalue  of  $\phi(g)$.\end{lemma}

\bp \edit{If $g$ is reducible in $G$ then the result follows from Corollary \ref{hh2}. Suppose that $g$ is irreducible in $G$, and hence $|g|$ divides $q^2+q+1$.
Observe that $p>2$ as $q^2+q+1$ is odd. If Sylow $p$-subgroups are not cyclic then
 $p=3$ and $3|(q-1)$, and then $g$ is reducible by \cite[Lemma 3.2]{Z4}  a contradiction.}

So Sylow $p$-subgroups are  cyclic.
If $\ell=0$ or $p$ then the result is a special case of \cite[Theorem 1.1]{Z2}, and the claim on eigenvalue 1
for $\ell=0$ is contained in  \cite[Corollary 1.3(4)]{Z2}. Let $\ell\neq p$.
According to \cite[Example 3.2(ii)]{TZ8}, either $\phi$ lifts to characteristic 0 and the result follows
from that for $\ell=0$, or $\ell$ divides $q^2+q+1$ and
 the Brauer character of $\phi$ coincides on the $\ell'$-elements  with $\tau -1_G$, where $\tau$
is the unipotent character of degree $q^2+q$ of  $G$. This again implies the result.  \enp

\begin{lemma}\label{bt4}
{\rm (Borel and Tits, see \cite[\S 13.1]{GL})}  Let $H $ be a finite reductive  group in characteristic $r$ and $g\in G$. If $g$ normalizes   an $r$-subgroup of $H$ then $g$ belongs to a parabolic
    subgroup of $H$. In particular, this holds if $g$ is not regular. \end{lemma}

Let $G$ be a finite quasi-simple group of Lie type in characteristic $r>0$.
 Let $\Phi_m(x)$ denote the cyclotomic \po for $m$-th roots of $1$,
and $\prod_m\Phi_m^{l_m}(x)$  a \po associated with $G$, see   \cite{GL}, pages $110--111$. Set $|G|_{r'}: =|G|/|U|$ where $U$ is a Sylow $r$-subgroup of $G$. Then $|G|_{r'}=\prod_m\Phi_m^{l_m}(q)$. If
$G={}^2B_2(q), {}^2F_4(q)$ we assume that $q=2^{2a+1}$, and if $G={}^2G_2(q)$ then $q=3^{2a+1}$ which notation agrees with that in \cite{GL}. Throughout this section $m_p$ denotes the  multiplicative order of $q (\bmod\ p)$, and  $e_p$ is the $p$-part of $\Phi_{m_p}(q)$. Observe that $ \Phi_1(q)=q-1,$ $ \Phi_2(q)=q+1,$
$\Phi_3(q)=q^2+q+1,$  $\Phi_4(q)=q^2+1,$ $\Phi_5(q)= q^4+q^3+q^2+q+1,$ $  \Phi_6(q)=q^2-q+1,$
$\Phi_8(q)=q^4+1,$ $\Phi_{10}(q)= q^4-q^3+q^2-q+1,$ $\Phi_{12}(q)=q^4-q^2+1.$

\begin{lemma}\label{sylow}
{\rm (\cite[\S 4.10.2]{GLS} and \cite{BM2})}
\begin{enumerate}[\rm(i)]
\item $|G|_{r'}=\prod_m\Phi_m^{l_m}(q)$;

\item Then $S$ is  cyclic if and only if there is exactly one $m$
such that $p$ divides  $\Phi_m(q)$ and $l_m=1$ for this $m.$

\item For every factor $\Phi_m (x)$ of the above \po there is a
torus $T$ of $G$ such that $|T|=\Phi_m^{l_m}(q)$. All tori of
order $\Phi_m^{l_m}(q)$ are conjugate in $G$. In addition, $T$ is
a  direct product of subtori of order $\Phi_m (q)$.

\item Let $m_p$ be the  multiplicative order of $q\pmod p$ and
let $T_p$ be a torus in $(iii)$ corresponding to $m=m_p$. Then
$N_G(T)$ contains a conjugate of $S$. Furthermore, if $S\subset
N_G(T)$ then the subgroup $A:=T\cap S$ is homocyclic of rank
$l_{m_p}$ and of exponent $e_p$. 
 $\hfill \Box$
\end{enumerate}
\end{lemma}

 \def\au{automorphism }

\begin{lemma}\label{cf6}
Let ${\mathbf G}$ be a simple simply connected algebraic
group of rank $n>0$, $\Fr$ a Frobenius endomorphism of ${\mathbf G}$, and $G:={\mathbf G}^{\Fr}$. Let A be as in Lemma {\rm \ref{sylow}(iv)}.

\begin{enumerate}[\rm(i)]
\item {\rm \cite[Proposition 4.8]{Z4}} Let $p>2$ be a prime dividing $|G|$ and $e_{p}=:|\Phi_{m_p}(q)|_p$, that is, $e_{p}$ is the exponent of $A$. Then every $p$-element $g\in G$ of order at most $e_{p}$ is conjugate to an element in $A$.

\item Let $\ep\in \{\pm 1\}$ be such that $4|(q-\ep)$, and let $q-\ep=2^em$, where
 $m$ is odd.  Suppose that $G$ has a maximal torus $T$ of order $(q-\ep)^n$. Then every $2$-element of $G$ of order at most $2^e$ is conjugate to an element of $T$.
\end{enumerate}
\end{lemma}

 \bp (ii) Let  $ {\mathbf T}$  be an $\Fr$-stable maximal torus of ${\mathbf G}$ such that $ T={\mathbf T}^{\Fr}$. Let $g\in G$ with $g^{2^e}=1$. It is well-known that $g$ is ${\mathbf G}$-conjugate  to an element $g'\in {\mathbf T}$.
 Set $T_2=\{t\in {\mathbf T}:t^{2^e}=1\}$.   Then $|T_2| = 2^{en}$. Therefore, $T_2$ coincides with the subgroup $\{x\in T :x^{2^e}=1\}$, so  $g'\in T$. As ${\mathbf G}$ is simply connected, $C_{\mathbf G}(g)$ is connected  \cite[Ch.II, 3.9]{SS}. By
\cite[Ch.I, 3.4]{SS}, the elements $g, g'\in G$ are conjugate
in $G$ provided $C_{{\mathbf G}}(g) $ is connected. So the claim follows. \enp


\section{Some observations on representations of groups of Lie type}

Recall that $\Irr(G)$ partitions into (rational) Lusztig series denoted by ${\mathcal E}_s$, where $s$ runs over the representatives of the conjugacy classes of semisimple elements of the dual group $G^*$.
The characters in ${\mathcal E}_1$ are called unipotent.

\bl{aa1} Let $T$ be a maximal torus of a finite reductive group $G=\GC^F$, and let $t_1,t_2\in T$ be regular elements.
If $t_1,t_2$ are conjugate in G then they are conjugate in $N_G(T)$.
\el

\bp Let $\mathbf{T}$ be the maximal torus of $\mathbf{G}$ containing $T$. Then $\mathbf{T}$ is unique and
$t_1,t_2$ are conjugate in $N_{\mathbf{G}}(\mathbf{T})$. Let $nt_1n\up=t_2$ with $n\in N_{\mathbf{G}}(\mathbf{T})$. Then $\Fr(n)t_1\Fr (n\up)=t_2$, whence $n\up \Fr (n)t_1n\Fr (n\up)=t_1$,
that is, $n\up \Fr (n)\in C_{\mathbf{G}}(t_1)=\mathbf{T}$. By the Lang theorem, $n\up \Fr (n)=t\up \Fr (t)$ for some $t\in \mathbf{T}$. So $tn\up=\Fr (t)\Fr (n\up)=\Fr (tn\up)$, so $tn\up\in G$ and $x:=nt\up\in G$. Clearly, $xt_1x\up=t_2$ and $x\in
N_{\mathbf{G}}(\mathbf{T})\cap G=N_{G}(\mathbf{T})$. As $T=\mathbf{T}\cap G$, we have $xTx\up=T$, as required.\enp

\begin{lemma}\label{sy1}
Let G be a finite  group of Lie type. Let $\chi$ be an \ir unipotent character of $G$, and $S$ a maximal torus of G. 
Suppose that every element $1\neq t\in S$ is regular (so G is simple). Then $\chi$ is constant on $S\smallsetminus \{1\}$ and $\chi(t)\in\{0,1,-1\}$. Equivalently,
  $\chi|_{S}=\frac{\chi(1)-\eta}{|S|}\cdot \rho_S^{\reg}+\eta\cdot 1_S$, where $\eta\in\{0,1,- 1\}$.
  \end{lemma}

\begin{proof} For a function $f$ of $G$ denote by $f^{\#}$ the restriction of $f$ to the set of semisimple elements of $G.$ By the Deligne-Lusztig theory, if $\chi\in \Irr(G)$ then $\chi^{\#}$ is a $\QQ$-linear combination of $R^{\#}_{T_i,\theta_i}$, where  $R_{T_i,\theta_i}$ are some Deligne-Lusztig characters,  $T_i$ is a maximal torus of $G$ and $\theta_i$ is a linear character of $T_i$. Let $a_i$ be the coefficient of $R^{\#}_{T_i,\theta_i}$ in the expression in question. The  values $R_{T_i,\theta_i}(h)$
at the semisimple elements $h\in G$ are given by the formula $R_{T_i,\theta_i}(h)=\ep(T_i)\ep(G)\theta_{i}^G(h)/\StC(h)$, where
$\StC$ is the Steinberg character of $G$ and $\ep(T_i),\ep(G)\in\{\pm 1\}$, see for instance \cite[Prop. 7.5.4]{C}. It is well known that a regular semisimple element
of $G$ lies in a unique maximal torus, so either $T_i$ is conjugate to $ S$ or
$R_{T_i,\theta_i}(t)=0$ for every $t\in (S\setminus\{1\})$.  Therefore,  we conclude that
either $\chi(t)=0$ for all  $t\in (S\setminus\{1\})$,  or $ \chi(t)=\sum a_iR_{S,\theta_i}(t)$, with some non-zero coefficient  $a_i$. (Hence $\chi=\sum a_i R_{S,\theta_i}+f$, where $f$ is a class function vanishing on $S\setminus\{1\}$.) Furthermore, $\StC(h)=\ep(T_i)$
whenever $h$ is regular and $h\in T_i$. So  $R_{S,\theta_i}(t)=\ep(G)\theta_{i}^G(t)$, and hence $\chi(t)=\ep(G)\cdot \sum a_i\theta_{i}^G(t)$. Furthermore, if $h,h'\in S$ are conjugate in $G$ and regular then $h,h'$ are conjugate in $N=N_G(S)$ by Lemma \ref{aa1}. As $C_G(t)=S$, it follows that $\theta_{i}^G(t)=\theta^{N}_{i}(t)$.

If  $\chi$ vanishes on $S\smallsetminus \{1\}$ then $\chi|_{S} $ is a multiple of the regular character $\rho_{S}^{\reg}$.

 Suppose that $\chi$ is unipotent. Then $\theta_i=1_{S}$ is the trivial character of $S$, so $\chi(t)=(\sum a_i)\ep(G)\cdot 1_{S}^G(t)=a\ep(G)|N/S|$, where $a=\sum a_i$. In particular, $\chi$ is  constant on $S\setminus\{1\}$.

 Let $p$ be a prime dividing $|S|$. Then   $S$ contains a \syl of $G$. As $C_G(t)=S$ for $1\neq t\in S$,   every $p$-singular element is conjugate to that in $S$. Therefore,
 $\chi$ is constant at the $p$-singular elements of $G$.
 Then, by \cite[Theorem 1.3]{PZ},  $\chi$  belongs to the principal $p$-block of $G$ and
 $\chi(t)=\eta\in\{\pm 1\}$. Therefore, $\chi(t)=\eta$   for every $1\neq t\in S$. So we are done in this case.\end{proof}

\begin{remar}
{\em Suppose that $\chi$ is not unipotent. Then $\theta_i\neq 1_{T_i}$. Then  $\theta_{i}^G(h)$ is the sum of $\theta_i(h')$, where $h'$ runs over all
elements of $ T_i$ that are conjugate to $h$. The number of them is $N_G(T_i)/T_i$ as $T_i$ is a TI-set, and this does not depend on the choice of $1\neq h\in T_i$.
Then  $\theta_{i}^G(h)$ is the sum of $|N_G(T_i)/T_i|$
non-trivial $|h|$-roots of unity.}
\end{remar}

\med
 An \ir Brauer character $\phi$ for $\ell$ different from the defining characteristic of $G$ is called unipotent if $\phi$
 is a constituent of $\chi^\circ$ for some unipotent ordinary character $\chi$. Let $G^*$ be the group dual to $G$. For  a semisimple $\ell'$-element   $s\in G^*$, denote by $\mathcal{E}_{\ell,s}$ the union of the sets $ \mathcal{E}_{ys}$, where  $y\in G^*$, $ys=sy$ and $|y|$ is an $\ell$-power. Then $\mathcal{E}_{\ell,s}$ is a union of $\ell$-blocks (\cite[Theorem 9.4.6]{Cr}), so for every  $\phi\in \Irr_{\ell} G$  there exists a semisimple $\ell'$-element $s\in G^*$ such that $\phi$
is a constituent of $\chi^\circ$ for some $\chi\in\mathcal{E}_{\ell,s}$. (Moreover, $\chi$ can be chosen in $\mathcal{E}_{s}$ \cite[Theorem 3.1]{H1}.) Therefore, it is meaningful to write $\phi\in\mathcal{E}_{\ell,s}$.

\bl{uu1}
Let $\phi\in\mathcal{E}_{\ell,s}$ be a Brauer character. Then the restriction of $\phi$ to semisimple $\ell'$-elements is a $\QQ$-linear combination
 of the ordinary characters of $\mathcal{E}_{s}$ restricted to semisimple $\ell'$-elements.
\el

\bp Let $\chi\in \mathcal{E}_{s}$ be such that  $\phi$ is a constituent of $\chi^\circ$.
Every \ir Brauer character of an $\ell$-block is a $\ZZ$-linear combination of the ordinary characters of this block restricted to $\ell'$-elements \cite[Lemma 3.16]{N}, so
$\phi$ is a $\ZZ$-linear combination of the ordinary characters of $\mathcal{E}_{\ell,s}$ restricted to  $\ell'$-elements.
In  turn, the ordinary characters of $\mathcal{E}_{\ell,s}$ restricted to the semisimple elements are $\QQ$-linear combinations  of the Deligne-Lusztig characters defining $\mathcal{E}_{\ell,s}$ restricted to semisimple elements (see \cite[Lemma 4.1]{TZ8}). Therefore,  $\phi$ is a  $\QQ$-linear combination of such Deligne-Lusztig characters restricted to semisimple $\ell'$-elements.
As every Deligne-Lusztig character from $ \mathcal{E}_{ys}$ for $y\in C_G(s)$, $ys\neq 1$, restricted to semisimple elements coincides with some
Deligne-Lusztig character from $ \mathcal{E}_{s}$ restricted to semisimple elements
\cite[Prop 2.2]{H1}, $\phi$  is a $\QQ$-linear combination 
 of the ordinary characters of $\mathcal{E}_{s}$ restricted to semisimple $\ell'$-elements,  the result follows.\enp

\begin{remar}
{\em There is a conjecture that $\phi$ is a  $\ZZ$-linear combination of the ordinary characters of $\mathcal{E}_{s}$ restricted to $\ell'$-elements. This has been proven for many cases, see \cite[Theorem 5.1]{GH} and \cite[Section 9]{Cr},
in particular, this is true if $G={}^2F_4(q)$ and ${}^3D_4(q)$, and if $G=SL_3(q)$, $\SU_3(q)$ for $\ell\nmid |Z(G)|$.}
\end{remar}

\begin{corol}\label{sy2} Under the assumptions of Lemma {\rm \ref{sy1}} 
let $\phi$ be a unipotent $\ell$-Brauer character of $G$. 
Then $\phi$ is constant on the set $S\smallsetminus \{1\}$.
\end{corol}

\begin{proof} This follows from Lemmas {\rm \ref{sy1}} and \ref{uu1}. \end{proof}

 The \f lemma refines Theorem 4.2 in \cite{TZ8}.

\bl{cs2}
Under the assumptions of Lemma {\rm \ref{sy1}} let $\phi\in \Irr_{\ell} G$ and $\phi\in {\mathcal E}_{\ell,s}$. 
Then one of the \f holds:

\begin{enumerate}[\rm(i)]
\item $s=1$, $\phi$ is unipotent and constant on the $\ell'$-elements of $S \setminus \{1\};$

\item $s\neq 1$, $|s|$ is coprime to $|S|$ and $\phi(t)=0$ for all $\ell'$-elements  $t\in (S \setminus \{1\});$

\item $s\neq 1$, $|s|$ divides $|S|$ and $\phi$ lifts to characteristic $0$.
\end{enumerate}
\el

\begin{proof} If $s=1$ then $\phi$ is unipotent, and we have (i) by Corollary \ref{sy2}.
Let $s\neq 1$.

 Suppose first that $|s|$ is coprime to $|S|$.
Let $\chi\in {\mathcal E}_{s}$. Then, on restriction to semisimple elements, $\chi$ agrees with a $\QQ$-linear combination of   the Deligne-Lusztig characters $R_{T_i,\theta_i}$ defining $\mathcal{E}_{s}$ (see \cite[Lemma 4.1]{TZ8}). Here, $|\theta_i|=|s|$ by \cite[Lemma 2.1(a)]{H1}, so $|s|$ divides $|T_i|$. Therefore, $T_i$ is not conjugate to
$S$, and hence $R_{T_i,\theta_i}(t)=0$ for every $1\neq t\in S$. So $\chi(s)=0$ for every $\chi\in {\mathcal E}_{s}$. Now (ii) follows by
Lemma \ref{uu1}.

Suppose that $|s|\neq 1$  divides $|S|$.  As $ys$ is a regular semisimple element for $y\in C_G(s)$, $ys\neq 1$, and hence $S=C_G(ys)$, the set ${\mathcal{E}_{ys}}$ consists of a single character of degree $d=|G|_{r'}/|S|$, where $r$ is the defining characteristic of $G$. In addition, ${\mathcal{E}_{\ell,s}}$ is a union of $\ell$-blocks, so, by \cite[Lemma 3.16]{N}, every \ir Brauer character $\phi$ in any of these blocks is a $\ZZ$-linear combination of ordinary characters of degree $d$, and $\phi$ itself is a constituent of $\eta^\circ$ for some irreducible character $\eta \in {\mathcal{E}_{\ell,s}}$, which is of degree $d$.
Hence $\phi(1)=d$ and $\phi=\eta^\circ$. \end{proof}

\section{Unipotent elements in $\GL_n(F)$, char$\,F=2$}

Let $1\neq g\in \GL_n(2)=\GL(V)$ be a 2-element and $z\in\lan g\ran$ an involution.
We set
$$j(g) =j(z):= \dim(\Id-z)V.$$
Note that $j(g)$ equals the number of blocks of size 2 in the Jordan canonical form of $z$.

\bl{3jn}
Let $g\in \GL_n(2)$ be an element of order $2^{m+1}$, $m>0$,  and $z:=g^{2^m}$.
Let $J(g)=(J_{n_1}\ld J_{n_k})$ be the Jordan canonical form of g. Set $l_i=\max(0,n_i-2^m)$ for $i=1\ld k$. Then $j(g)=\sum l_i$.\el

\begin{proof} It suffices to prove the statement in the case $k=1$, where we have
$$j(g)=\dim(\Id-z)V=\dim(\Id-g)^{2^m}(V)=n-2^m=l_1.$$
\end{proof}

\bl{2yy}
Let $0<c<d<n$ be integers, and $n=kd+l$ for $0\leq l<d$. For a sequence $\lam=(n_1\geq \cdots\geq  n_k\geq 0)$ set $\overline{\lam}=(l_1\ld l_k)$  where
 $l_i=\max(0,n_i-c)$ for $i=1\ld k$. Suppose that $\sum n_i=n$. Then
$$\sum _il_i\leq k(d-c)+\max(0,l-c).$$
In addition, if $d<d'<n$ then
$$k(d-c)+\max(0,l-c)\leq k'(d'-c)+\max(0,l'-c),$$
where $n = k'd'+l'$ with $0 \leq l' < d'$.\el

\begin{proof} This becomes clear if one views $\lam$ as a Young diagram of size $n$, and let $\mu=(d\ld d,l)$, where $d$ is repeated $k$ times. Then $\mu$ can be obtained from $\lam$ by moving down certain boxes of $\lam$ (note that $n_1\geq d$). In addition, $\overline{\lam}$ is a Young diagram of size $\sum _il_i$, obtained from $\lam$ by deleting the first $c$ columns.   (Note that $\max(0,n_1-c)\geq \cdots \geq \max(0,n_k-c)\geq 0$.)
The first assertion in the lemma means that  the size of $\overline{\lam}$ is not greater than the size of $\overline{\mu}=(d-c\ld d-c,\max(0,l-c))$,  which is again obtained from $\overline{\lam}$
 by removing the first $c$ columns. It is clear that the number of boxes removed from $\mu$ to obtain $\bar\mu$ is at least the number of boxes removed from $\lam$ to obtain
$\bar\lam$, whence the assertion follows. Moreover, this number of removed boxes does not increase if one uses $d'$ instead of $d$ to form $\mu$, whence the second assertion follows.
\end{proof}

\bl{77n}
Let $g\in \GL_n(2)$ be an element of order $|g|=2^{m+1}$, $m> 0$.
Let $d$ be the minimum \po degree of g.
 Suppose that   $d<|g|$. Then
$$j(g)\leq  (n-l)(1-\frac{|g|}{2d})+\max(0,l-\frac{|g|}{2}),$$
where $n\equiv l\pmod{d}$ and $0\leq l<d$. If $d=|g|-1$ then $1-\frac{|g|}{2d}=\frac{|g|-2}{2(|g|-1)}$.

\el

\begin{proof}
We first show that the bound is attained. Write 
$n=kd+l$ with $k\geq 0$. Let $x=\diag(J_d\ld J_d,J_l)$, where $J_d$ occurs $k$ times. By Lemma \ref{3jn},
 $j(J_l)=
\max(0,l-\frac{|g|}{2})$ and $j(J_d) =d-\frac{|g|}{2}$. So $j(g)=k(d-\frac{|g|}{2})+\max(0,l-\frac{|g|}{2})$ and $k=\frac{n-l}{d}$. Then $j(g)=\frac{n-l}{d}(d-\frac{|g|}{2})+\max(0,l-\frac{|g|}{2})$, as claimed.

The Jordan form of  unipotent elements  $g\in \GL_n(2)$ can be encoded by
the Young diagrams, that is, $(n_1\ld n_k)$ corresponds to  $\diag(J_{n_1}\ld J_{n_k})$, where we assume $n_1\geq\cdots\geq n_k $, so $n_1$ is the minimum \po degree of $g$.
Then the inequality in the lemma follows from Lemma \ref{2yy} above. Indeed, if $g$ is as in the statement then  $n_1=d>c=:\frac{|g|}{2}$. \end{proof}

We specify the result of Lemma \ref{77n} as follows to make it more convenient for use in the next section.

\bl{fr0}
Let $g\in \GL_n(2)$ be a $2$-element.

\begin{enumerate}[\rm(i)]
\item Suppose that $|g|\leq (q+1)/2$. 
$$ j(g)\leq \frac{n(q-3)}{2(q-1)}+\frac{q-7}{4}.$$

\item Suppose that $|g| = q+1$. Then we have
$$j(g)\leq\begin{cases}  \frac{n(q-1)}{2q}&if \,\, q|n;\\
\frac{(n-1)(q-1)}{2q}&if \,\, q|(n-1);\\
\frac{(n+1)(q-1)}{2q}-1&if \,\,  q|(n+1).
\end{cases}$$
\end{enumerate}
\el

\begin{proof} (i) We have $|g|\leq (q+1)/2$ and $l-\frac{|g|}{2}\leq \frac{|g|}{2}-2 \leq \frac{q+1}{4}-2=(q-7)/4$.
By Lemma \ref{77n}, we have
$j(g)\leq n(1-\frac{|g|}{2d})+\frac{q-7}{4}$. Note that
$1-\frac{|g|}{2d}\leq 1-\frac{|g|}{2(|g|-1)}$ as $d\leq |g|-1$.

Furthermore, $|g|\leq (q+1)/2$ implies
$$1-\frac{|g|}{2(|g|-1)}=1-\frac{1}{2- (2/|g|)}\leq 1-\frac{1}{2-(2/(q+1)}=1-\frac{q+1}{2q-2}=\frac{q-3}{2(q-1)}.$$
So
$$j(g)\leq \frac{n(q-3)}{2(q-1)}+\frac{q-7}{4}.$$

\smallskip
(ii) Let $|g|=q+1=2^{m+1}$ and $d=|g|-1=q$. 

(a) Let $q|n$. Then $j(g)\leq\frac{n(|g|-2)}{2(|g|-1)}=\frac{n(q-1)}{2q}$.

(b) Let $q|(n-1)$. Then $j(g)\leq \frac{(n-1)(|g|-2)}{2(|g|-1)}=\frac{(n-1)(q-1)}{2q}$.

(c) Let $q|(n+1)$. Then $l=d-1=|g|-2=q-1$ so
$$j(g)\leq\frac{(n-q+1)(|g|-2)}{2(|g|-1)} +|g|-2-(|g|/2)=\frac{(n+1)(q-1)}{2q}-1.$$
\end{proof}

\section{The case of $\SU_3(q)$}

In this section we refine our results on minimal polynomials of elements of the group $\SU_3(q)$ in its cross-characteristic \ir representations. The main result is Proposition \ref{pr5}.

 \begin{lemma}\label{cv8}
 Let $S=\SU_3(q)\leq G \leq H=U_3(q)$, $2 \nmid q$, and let $g\in G $ be a non-central
 semisimple $2$-element, and let $|g|=2^\al$.
Let $\phi\in{\rm IBr}_{\ell}(G)$ with $(\ell,q)=1$ and $\dim\phi>1$. Suppose that $\deg\phi(g)<o(g)$.
Then one of the \f holds:
\begin{enumerate}[\rm(i)]
\item $\ell=2$,  g is not a pseudoreflection and $|g|$ divides $q+1;$

\item $q+1=2^\al$, $g$ is a pseudoreflection and $\deg\phi=o(g)-1$,

\item  $q+1=2^{\al-1}$, $g^2$ is a pseudoreflection and $\deg\phi=o(g)-2.$

\item  $|g|=q+1=4$, g is not a pseudoreflection and $\deg\phi=3;$
\end{enumerate}
\noindent In addition, in cases {\rm (ii)}, {\rm (iii)} and {\rm (iv)}, $\Phi$ is a Weil \rep of G.
\el

\begin{proof} Suppose first that $g$ is contained in a parabolic subgroup of $G$.
 Then, applying the main result of \cite[Theorem 13.2]{DZ1}
 and \cite[Theorem 3.2]{GMST}, we conclude
that  (ii) or (iii) holds.

Suppose that $g$ is not contained in any parabolic subgroup of $G$. Note that $g$ is contained in a
maximal torus $T$ of $H$, and $|T|\in\{ q^3+1, (q+1)(q^2-1), (q+1)^3\}$. The tori of order
$ (q+1)(q^2-1)$ lie in parabolic subgroups, and those of order $ q^3+1$ contains no non-central 2-element. The torus of order $(q+1)^3$ is of exponent $q+1$, so $|g|$ divides $q+1$. If $\ell\neq 2$ or  $|g|=q+1=4$ then the argument of the proof \cite[Lemma 6.1]{TZ8} works (see page 653 there). So we have (i) and (iv). \end{proof}


\med
 We improve the conclusion in (i) of Lemma $\ref{cv8}$ as follows: 

\begin{propo}\label{pr5}
Let $G=\SU_3(q)$, q odd, and let $g\in G$ be a $2$-element of order dividing $q+1$. Let $\phi$ be a non-trivial \ir $2$-modular \rep of G.
Then the minimum \po of  $\phi(g)$ 
is of degree $|g|$, unless possibly $\dim \phi=q(q-1)$,  $|g|=q+1$ and $\deg\phi(g)=|g|-1$.
If $\deg\phi(g)=|g|-1$ then the Jordan form of $\phi(g)$ consists of  $q-1$ blocks of size q.
\end{propo}

Before proving Proposition \ref{pr5}, we deduce a consequence of it:

\begin{corol}\label{cr1}
Let $G=\SU_3(q)$, and let $g\in G$ be a $p$-element, $p \nmid q$.
 Let $o(g)$ be the order of g modulo $Z(H)$ and let $\phi$ be a non-trivial \ir $\ell$-modular \rep of $G$, $\dim\phi>1$. Then $\deg\phi(g)=o(g)$, unless, possibly, $\phi$ is a Weil \rep of $G$.
\end{corol}

\bp Note that $o(g)=|g|$ if $(p,q+1)\neq 3$, in particular, if $p\neq 3$.
 The result is contained in Proposition \ref{pr5} if $p=2$ (as $\dim\phi=q^2-q$ implies  $\phi$ to be Weil), and in \cite[Lemma 6.1]{TZ8} for $p>2$.\enp

The proof of Proposition \ref{pr5} occupies the rest of this section. This is trivial if $|g|=2$ so we assume $|g|\geq 4$. Then $4|(q+1)$. We start with some elementary observations. \med

Note that  the involutions in $G=\SU_3(q)$ are conjugate. Denote by $z$ an involution from
a parabolic subgroup $P$ of $G$. Let $U$ be the unipotent radical of $P$. We can write
\med

\centerline{$z=\begin{pmatrix} -1&0&0\\ 0&1&0\\ 0&0&-1\end{pmatrix}$ \,\, $U=\left\{u=\begin{pmatrix} 1&a&c\\ 0&1&b\\ 0&0&1\end{pmatrix}\right\},$ and $zuz\up=\begin{pmatrix} 1&-a&c\\ 0&1&-b\\ 0&0&1\end{pmatrix}$}\med
\noindent where $a\in \FF^\times _{q^2}$, $c\in \FF^\times _{q}$ and $b=a^q$. So $|U|=q^3$, $[z,Z(U)]=1$. If $\psi$ is an \irr of $U$ of dimension not 1 then $\dim \phi=q$,  the \eis of $\phi(z)$
are $1,-1$, and their multiplicities are $(q-1)/2$ and $(q+1)/2$, with the \mult of $-1$ being even. As $4|(q+1)$,  we conclude that the \mult of 1 is $(q-1)/2$. In fact, if $\psi$ is  2-modular then
the Jordan form of $\phi(z)$ has $(q-1)/2$ blocks of size 2.

\bl{nb2}
Let $\phi$ be an \ir $2$-modular Brauer character of $G=\SU_3(q)$, $4|(q+1)$.
 Let  $z\in G$ be an involution.
Denote by $j(\phi(z))$ the number of non-trivial Jordan blocks of $\phi(z)$. Then
$$j(\phi(z))\geq \frac{(q-1)\bigl(\phi(1)-(\phi|_{Z(U)},1_{Z(U)})\bigr)}{2q}+
\frac{(\phi|_{Z(U)},1_{Z(U)})-(\phi|_{U},1_{U})}{2}.$$
\el

\begin{proof} Let $ V$ be the underlying module of $\phi$. Then $V|_U=V^U\oplus V_1\oplus V_2$, where $V_1+V^U=V^{Z(U)}$
and $V_2 = [V,Z(U)]$. So $\dim V_2=\dim V-\dim V^{Z(U)}$. Then $V_2$ is the direct sum of \ir $U$-modules
non-trivial on $Z(U)$, so  $V_2$ is the sum of $\frac{\dim V_2}{q}$ \ir $U$-modules of dimension $q$. In addition, $V_1$ is the sum of non-trivial one-dimensional $U$-modules. As $C_{U}(z)=Z(U)$, it follows that no non-trivial linear character of $U$ is $z$-invariant; this implies the number of Jordan blocks of $z$ on $V_1$ to be equal to $\dim V_1/2$. Therefore, the number in question equals

$$j(\phi(z) )=\frac{(q-1)\dim V_2}{2q}+\frac{\dim V_1}{2}.$$
As $\dim V^U=(\phi|_{U},1_{U})$, $\dim V^{Z(U)}= (\phi|_{Z(U)},1_{Z(U)})$, and $\dim V_1=\dim V^{Z(U)}-\dim V^U$, the statement follows.
\end{proof}

The terms of the formula in the lemma can be easily computed by using the character table of $G$ and the decomposition matrix of $G$ modulo 2. This is known to experts but the result is not explicitly written
in literature. So we provide some detailed comments below.

There are three unipotent characters of $G$, of degree 1, $q^3$ and $q^2-q$ \cite{Ge}. The latter is \ir modulo 2
\cite[Proposition 9]{HM}, and  $\chi^\circ_{q^3}$ decomposes  as $1_G+2\chi^\circ_{q^2-q}+\phi_0$, where  $\phi_0$ is \ir of degree $q^3-2q^2+2q-1$ \cite[Theorem 4.1]{H4}. This also shows that all these unipotent characters are in the principal block $B_0$. and they form a basic set for the union
$\mathcal{E}_2(G,1)$ of all $\mathcal{E}(G,s)$ with $s$ a $2$-element by \cite[Theorem A]{Ge3}. It follows that this union is precisely $B_0$, and $B_0$ contains $3$ irreducible Brauer characters.

Now we consider $\chi\in \mathcal {E}_2(G,s)$, where $s \in G^* = {\rm PGU}_3(q)$ is semisimple of odd order $|s| > 1$.  Again by \cite[Theorem A]{Ge3}, the characters in $\mathcal{E}(G,s)$ form a basic set for $\mathcal {E}_2(G,s)$. The element $s$ belongs to a maximal torus $T*$ of $G^*$, of order $(q+1)^2$, $q^2-1$, or $q^2-q+1$.

\medskip
Suppose first that $|T^*|=q^2-q+1$. Then $s$ is regular in $G^*$. If $|s|>3$ then ${\mathcal E}(G,s)$ consists of a unique character of
degree $(q+1)(q^2-1)$. If $|s|=3$ then $3|(q+1)$ and ${\mathcal E}(G,s)$ consists of three characters of degree $(q+1)(q^2-1)/3$. These characters are all of $2$-defect 0.

\medskip
Suppose next that $|T^*|=q^2-1$ but $s$ does not belong to any torus of order $(q+1)^2$. Then $s$ is regular 
and $|C_{G^*}(s)|=q^2-1$. So ${\mathcal E}(G,s)$ consists of a unique character of degree
$q^3+1$. As every Brauer character is an integral linear combination of ordinary characters in its block, we conclude that its degree is $q^3+1$ too, and it is liftable.

\medskip
Suppose now  that $|T^*| = (q+1)^2$. Then $|s|$ divides  $q+1$  and  one of the \f holds:

\smallskip
(i) $s$ is regular. If $|s|>3$ then $|C_{G^*}(s)|=(q+1)^2$ and ${\mathcal E}(G,s)$ consists of
a unique character of degree $(q-1)(q^2-q+1)$, hence, as above, the unique Brauer character in the block has the same degree. If
$|s|=3$, then $|C_{G^*}(s)|=3(q+1)^2$, and $\EC(G,s)$ consists of $3$ characters $\chi_{1,2,3}$ of the same degree
$(q-1)(q^2-q+1)/3$. It follows that every Brauer character of ${\mathcal E}_{2}(G,s)$ is of degree divisible
by $(q-1)(q^2-q+1)/3$. Hence, each $\chi_i^\circ$ is irreducible, and, as $\chi_{1},\chi_2,\chi_3$ form a basic set, we again see that each irreducible Brauer character is liftable.

\smallskip
(ii)  $s$ is not regular. Then
$C_{G^*}(s)\cong {\rm GU}_2(q)$, and $\EC(G,s)$ consists of two characters, $\chi_1$ (a Weil character) of degree $q^2-q+1$, and $\chi_2$ of degree $q(q^2-q+1)$.
It is well known that $\chi_1^\circ$ is irreducible, see e.g. \cite[Proposition 9]{HM}.
We can represent $s$ by a diagonal matrix $\diag(\alpha,\alpha,1)$ in ${\rm GU}_3(q)$ with $\alpha \neq 1$ of odd order dividing $q+1$.
Then the element $t \in G^*$ represented by $\diag(\alpha,-\alpha,1)$ centralizes $s$ and has $s$ as its $2'$-part, and
$\EC(G,t)$ consists of a unique character $\psi$ of degree $(q-1)(q^2-q+1)$.
Note that $\psi(t) = 2q-1$, $\chi_1(t)=1-q$, and $\chi_2(t)=q$ for a transvection $t \in G$, see \cite[Table 3.1]{Ge}. Since
$\psi^\circ$ is a linear combination of $\chi_1^\circ$ and $\chi_2^\circ$, it follows that
$\chi_2^\circ=\psi^\circ+\chi_1^\circ$, and $\{\psi^\circ,\chi_1^\circ\}$ is a basic set for $\EC_2(G,s)$.
We claim that $\psi^\circ$ is irreducible. Indeed, $\psi^\circ = a\gamma+b\chi_1^\circ$ for some irreducible Brauer character
$\gamma$ and some integers $a,b \geq 0$. Inspecting the multiplicity of $1_U$ and of any nontrivial linear character $\xi$ of
$U$, using \cite[Table 3.2]{Ge}, we obtain
$(\psi|_U,1_U)=0=a(\gamma|_U,1_U)+b$ and $(\psi|_U,\xi)=1= a(\gamma|_U,\xi)$, whence $(a,b)=(1,0)$, i.e. $\psi^\circ=\gamma$, as
claimed. Therefore, there are two irreducible Brauer characters in $\EC_2(G,s)$, of degree $q^2-q+1$ and
$(q-1)(q^2-q+1)$, and they both lift.

Thus, $\phi_0$ is  the only non-liftable 2-modular \ir Brauer character, and it is  of degree $q^3-1-2(q^2-q)$.

Let $1\neq u\in Z(U)$ and $v\in (U\setminus Z(U))$. Then $\phi(u),\phi(v)$ do not depend on the choice of $u,v$. Then
$$(\phi|_{Z(U)},1_{Z(U)})=\frac{1}{q}(\phi(1)+(q-1)\phi(u))\,\,
{\rm and} \,\,(\phi|_U,1_{U})=\frac{1}{q^3}(\phi(1)+\phi(v)(q^3-q)+\phi(u)(q-1)).$$

For our purpose we could ignore \ir \reps  of degree $(q+1)(q^2-1)$ as these are of 2-defect 0,
so the restrictions of them to the Sylow 2-subgroup of $G$ are the characters of projective modules.

\smallskip
Next,  for every non-trivial \ir 2-modular Brauer character $\phi$ of $G$ of non-zero defect we  compute the \mult of $1_{Z(U)}$ in $\phi|_{Z(U)}$ and the \mult of $1_{U}$ in $\phi|_{U}$. This can be easily done by using the character table of $G$. The results are summarized in Table 1.

\smallskip
\centerline{Table 1}

\med
\noindent{\small
\begin{tabular}{|l|c|c|c|c| c|c|}
\hline
$\phi$   & $\phi(1) $ &$\phi(u)$ & $(\phi|_{Z(U)},1_{Z(U)})$& $\phi(v)$& $(\phi|_U,1_{U})$& $(\phi|_{Z(U)},1_{Z(U)})-(\phi|_{U},1_{U})$ \\
\hline
$\phi_1$& $q^2-q$ &  $-q$   & $ 0 $&0&0&0
   \\
\hline
$\phi_2$& $ q^2-q+1$ & $-q+1$   &  $1$ &1&0&1\\
\hline
$\phi_4$& $ (q-1)(q^{2}-q+1)$ & $2q-1$   &$ q^2-1$ &$-1$&0& $q^2-1$ \\
\hline
$\phi^{*}_4$& $ (q-1)(q^{2}-q+1)/3$ & $(2q-1)/3$   &$ (q^2-1)/3$ &$x,y,y  $&0& $(q^2-1)/3$ \\
\hline
$\phi_5$& $ q^{3}+1$ & $1$   &$ q^2+1$&1&2&$q^2-1  $ \\
\hline
$\phi_7$& $ q^{3}$ & $0$   &$ q^2$&0&1&$q^2-1  $ \\
\hline
$\phi_0$& $ q^{3}-2q^2+2q-1$ & $2q-1$   &$ q^2-1$&$-1$&0&$q^2-1  $ \\
\hline
\end{tabular}}

\bigskip

In Table 1, $x=(2q-1)/3,y=-(q+1)/3$. Note that the character $\phi^*_{4}$ exists \ii $3|(q+1)$; in this case the set $U\setminus Z(U)$ is the union of   three conjugacy classes, 
 and $\phi^*_{4}(v)=y$ for $v$ in two of them, and $\phi^*_{4}(v)=x$ when $v$ lies in the remaining class.
 In fact, each \irr of $U_3(q)$ of degree  $ (q-1)(q^{2}-q+1)$ restricts to $G$ as the sum of three \ir \reps of degree $ (q-1)(q^{2}-q+1)/3$. If $\phi_{4a},\phi_{4b},\phi_{4c}$ are the characters of these \reps of $G$ and $v_1,v_2,v_3$ are representatives of the three conjugacy classes  in question then the corresponding fragment of the Brauer character table is

 \med
 \begin{center}
\begin{tabular}{|l|c|c|c|}
\hline
        & $v_{1} $ &$v_{2}$ & $v_3$  \\
\hline
$\phi_{4a}$& $x$ &  $y$   & $ y $    \\
\hline
$\phi_{4b}$& $ y$ & $x$   &  $y$ \\
\hline
$\phi_{4c}$& $ y$ & $y$   &$x$    \\
\hline
 \end{tabular}
 \end{center}
 \med

 \noindent This is irrelevant for computation of $(\phi_4^*|_U,1_U)$. On the other hand, if $1_U$ does not occur as a constituent of $\phi_4|_U$ then $1_U$ does not occur as a constituent of $\phi^*_4|_U$.

\med
Let

$$f_1(n)=\frac{n(q-3)}{2(q-1)}+\frac{q-7}{4}\,\,\,\,{\rm and}\,\,\,\,
 f_2(n)=\begin{cases}n(q-1)/2q&if\,\,n\equiv 0\pmod q\\
(n-1)(q-1)/2q&if\,\,n\equiv 1\pmod q\\
 \frac{(n+1)(q-1)}{2q}-1&if\,\,n\equiv -1\pmod q.\end{cases}$$ be the functions defined in Lemma \ref{fr0}(i), (ii), respectively. Then we have
 (where $n=\phi(1)$):

\med
\centerline{Table 2}
\begin{center}
\noindent{\small
\begin{tabular}{|l|c|c|c|c| }
\hline
    $\phi$   & $n=\phi(1) $ &$j(\phi(z))\geq $&$f_1(n)$  &$f_2(n) $\\
\hline
$\phi_1$& $q^2-q$ &  $(q-1)^2/2$&$(2q^2-5q-7)/4$&$(q-1)^2/2 $
   \\
\hline
$\phi_2$& $ q^2-q+1$ & $(q-1)^2/2$ &$(2q^3-7q^2+1)/4(q-1)$&\\ 
\hline
$\phi_4$& $ (q-1)(q^{2}-q+1)$ & $(q-1)(q^2-2q+3)/2 $&$(2q^3-8q^2+9q-13)/4$& $-1+\frac{(q-1)(q^2-2q+2)}{2}$ \\
\hline
$\phi_{4}^*$& $ (q-1)(q^{2}-q+1)/3$ & $(q-1)(q^2-2q+3)/6 $   &$q^3-7q^2+5q-12)/6$&\\
\hline
$\phi_5$& $ q^{3}+1$ & $(q-1)(q^2+1)/2 $ &$\frac{2q^4-6q^3+q^2-6q+1}{4(q-1)}$&$ q^2(q-1)/2 $\\
\hline
$\phi_0$& $(q-1)(q^{2}-q+1)$ & $(q-1)(q^2-2q+3)/2   $ &$(2q^3-8q^2+9q-13)/4$&$-1+\frac{(q-1)(q^2-2q+2)}{2}$\\
\hline
\end{tabular}}
\end{center}

\med
In Table 2 we have left blank certain positions in the fifth column as this column is created under assumption that $|g|=q+1$. Recall that if $q+1$ is a 2-power then there are no \ir 2-modular \reps of degree $q^2-q+1$ and $(q-1)(q^2+q+1)/3$. Note that the third column gives the lower bound for $j(\phi(z))$ from Lemma  \ref{nb2} obtained using Table 1.

\begin{proof}[Proof of Proposition {\rm \ref{pr5}}] Let $d=\deg\phi(g)$. Suppose the contrary, that $d<|g|$. Let $n=\dim \phi$. Suppose first that $|g|\neq q+1$. Then $|g|\leq (q+1)/2$. Let $z$ be the involution in $\lan g \ran$, so $j(\phi(g))$
is the number of  non-trivial blocks in the Jordan form of $\phi(z)$. By Lemma \ref{nb2},
$j(\phi(g))\geq t(n)$, where $t(n)$ is given in the 3-rd column of Table 2.

By Lemma \ref{fr0}, $d<|g|$ implies $j(\phi(g))\leq f_1(n)$, so $t(n)\leq f_1(g)$.
One easily observes that this is false, whence a contradiction.

Let $|g|= q+1$.  Inspecting  the second column of Table 2, one observes that  $\phi_i(1)\pmod q \in\{1,0,-1\}$ for $i=1,4,5,6$. So we can use Lemma \ref{fr0}(ii) to build the upper bound for $j(\phi(g))$, which is  written in the 5-th column there. As above, we compare the entries of the 3rd and 5th columns of Table 2, to observe that these are not compatible for $i=4,5,0$.  (There is no contradiction for $i=1$.)

Let $n=\phi(1)=q^2-q$. Then $\GL_n(2)$ contains the matrix $J:=(J_q\ld J_q)$, where $J_q$
is repeated $q-1$ times. Then $j(J)=(q-1)^2$ by Lemma \ref{2yy}. Let $(J_{n_1}\ld J_{n_k})$
be the Jordan form of $\phi(g)$, where $n_1\geq \cdots\geq n_k$. If this is not $J$ then, by Lemma \ref{2yy}, $j(\phi(g))<(q-1)^2$, which is a contradiction (by Table 2). \end{proof}


\section{Groups of BN-pair rank $1$}

\subsection{Groups ${}^2B_2(2^{2m+1})$ and ${}^2G_2(3^{2m+1})$}

\begin{lemma}\label{sz1}
Let $G={}^2B_2(2^{2m+1})$, $m>0$, be a Suzuki group.
For $p>2$ let $g\in G$ be a  $p$-element.  Let $1_G\neq \phi\in\Irr_\ell G$ and $\ell \neq 2$.
Then $\deg\phi (g)=|g|$.\end{lemma}

 \bp Note that Sylow $p$-subgroups of $G$ are cyclic. If $\ell=0$ or $\ell=p$, the result follows from \cite{Z2}, see also \cite[2.8]{z90} for $\ell=0$. Suppose $\ell\neq 0,p$. Let $\beta$ be the Brauer character of $\phi$. If $\beta$ is liftable, the result follows from
\cite{Z2}. The decomposition numbers of $G$ have been determined by Burkhardt \cite{Bu2}.
Inspection of them in \cite{Bu2} shows that $\phi$ either liftable or
$\StC(x)-1=\phi(x)$ for every $l'$-element $x\in G$.
As $\StC(x)\in\{\pm 1\}$ and $\StC(1)=2^{2(2m+1)}$, the result easily follows from Lemma \ref{tr1}.\enp

\begin{lemma}\label{re1}
Let $G={}^2G_2(q)$, $q=3^{2m+1}$, $m>0$. Let $g\in G$ be a  $p$-element for some prime  $p\neq 3$ dividing
$|G|$. Let $\ell \neq 3$, $\phi\in\Irr_\ell G$, $\dim\phi>1$. Then $\deg\phi (g)=|g|$. \end{lemma}

Proof. The lemma is trivial for $p=2$, as every $2$-element of $G$ is an involution.

Let $p>2$. Then Sylow $p$-subgroups of $G$ are cyclic. Let $\beta$ be the
Brauer character of $\phi$. If $\ell\in\{0,p\}$ or $\beta$ is liftable, the result
follows from \cite{Z2}. Suppose otherwise.

 Let $\ell=2$.  By \cite[p.104]{LM}, $\be$ lies in the principal block and $\be\in\{\be_1,\be_2\}$, where
 $\be_2=\xi_2 - 1_G$ and  $\be_3=\StC^\circ+1_G -2\xi^\circ_2-\xi^\circ_6-\xi^\circ_8$ in notation of \cite{LM}.

Let $\ell> 2$. Then the Sylow $\ell$-subgroups of $G$ are cyclic. The Brauer tree for
$G$ for every $\ell>2$ is determined by Hiss \cite{H3}. Inspection in \cite{H3} shows that $\be$ either liftable or $ \ell|(3^{2m+1}+1)$ and
$\beta_1=\StC^\circ-1_G-\xi^\circ_5-\xi^\circ_7$,  or $ \ell|(3^{2m+1}-3^{m+1}+1)$ and
$\beta_4=\StC^\circ-1_G$, in notation of \cite{H3,W}.

It suffices to show that the restriction $\beta_i|_C-\rho^{\reg}_{C}$ is either 0 or a proper character of $C$
for every maximal cyclic $p$-subgroup $C$ of $G$. There are 4 maximal tori of $G$; these are cyclic groups of order  $|T_1|=q-1$, $|T_2|=q+1$, $|T_3|=q+\sqrt{3q}+1$ and $|T_4|=q-\sqrt{3q}+1$. We write $C=C_i$ if $|C|$
divides $|T_i|$.

Inspection of the character table of $G$ in \cite{W}
shows that every character $\xi_j$ $(j=2,5,6,7,8)$ as well as $\StC$
is constant at $C\smallsetminus \{1\}$. Therefore, the result follows
 by applying Lemma \ref{tr1} to the  values of these characters (given in \cite{W}).

 For reader's convenience in the \f table
 we give the values of the characters involved at $1\neq t\in C_i$ for $i=1,2,3,4.$

 \med
\begin{center}
\begin{tabular}{|c|c|c|c|c|c|c|c|c|c|c|c|}\hline
&$\StC$&$\xi_2$&$\xi_4$ &$\xi_5$ &$\xi_6$ &$\xi_7$&$\xi_8$&$\be_1$&$\be_2$&$\be_3$&$\be_4$\cr
\hline
$C_1$& 1& 1&1&0& 0&0&0&0&0&0&0\cr \hline
$C_2$&$- 1$&$3$ &$-3$&1 &$-1$ &1&$-1$&$-4$&2&$-4$&$-2$\cr \hline
$C_3$& $- 1$& 0&0&$-1$ &0 &$-1$&0&0&$-1$&0&$-2$\cr \hline
$C_4$& $- 1$&0&0 &$0$ &1 &0&1&$-2$&$-1$&$-2$&$-2$\cr
\hline
\end{tabular}
\end{center}

\med
 Note that $\xi_4(1)=q^2(q^4-q^2+1)$,  $\xi_2(1)=q^2-q+1$
and $\xi_6(1)=\xi_8(1)=(q^2-1)q(q^2+1-q\sqrt{3})/2\sqrt{3}$. So
 $\beta_2(1)=\xi_4(1)-\xi_2(1)-\xi_6(1)-\xi_8(1)=((q^6-q^2-q+1)+q(q^2-1)(q^2+1-q\sqrt(3))/\sqrt{3})$.

\section{The case of $G_2(q)$}

In this section we prove Theorem \ref{mth1} for $G=G_2(q)$, $q>2$, and $\ell\neq 2$. The case of $G=G_2(2)$
is considered in Lemma \ref{gg2} at the end of this section. (Also see Lemma \ref{2g24}  for group $2\cdot G_2(4)$.)

\begin{theo}\label{g2o}
Theorem {\rm \ref{mth1}} is true for $G=G_2(q)$, $q=r^a>2$.
 \end{theo}

\bl{mt1}
Let  $G=G_2(q)$ and $q \geq 3$. Then every semisimple element of G is contained in a subgroup isomorphic to $\SL_3(q)$ or $\SU_3(q)$.\el

\bp Every semisimple element is contained in a maximal torus of $G$.
Also, $G$ has 6 conjugacy classes of maximal tori, whose orders are $q^2-q+1$, $q^2+q+1$, $(q-1)^2$, $(q+1)^2$ and two classes of order $q^2-1$. In addition,  $G$ contains subgroups $L^+ \cong \SL_3(q)$ and $L^- \cong \SU_3(q)$, whose maximal tori are maximal in $G$.
If $3|q$, then the statement follows from fusion of conjugacy classes of certain subgroups $L^+\cong \SL_3(q)$ and
$L^-\cong \SU_3(q)$ in $G$ as given in \cite{En1}: every semisimple class of $G$ intersects $L^+$ or $L^-$.

Assume now that $3 \nmid q$ and choose $\ep\in\{+,-\}$ such that $3|(q-\ep)$.
It suffices to show that maximal tori $T^+ < L^+$ and $T^- < L^-$ (both cyclic groups of order $q^2-1$)
are not conjugate in $G$. Assume the contrary: $T^+$ and $T^-$ are conjugate in $G$. It is shown in \cite{CR} and \cite{EY2} that $G$ has two conjugacy classes of elements of order $3$, with representatives  $u,v$, where $C_G(u) \cong \SL^\ep_3(q)=L^\ep$, and $C_G(v) \cong \GL^\ep_2(q)$. Then $T^\ep$ contains
 $Z(L^\ep)$, and so   $u \in T^\ep$. As $T^{-\ep}$ is conjugate to $T^\ep$,
we may assume that $u \in T^{-\ep} < L^{-\ep}$.

Consider the case with $\ep=-$ and $\chi \in \Irr(G)$ of degree $q^3-1$. Then $\chi(u) =-q(q-1)$, see \cite{CR} and \cite{EY2}. On the other hand, if $\theta \in \Irr \SL_3(q)$ and $\theta(1) \leq q^3-1$, then $\theta(g) \geq 0$, or $\theta(1)=q^3-1$ and $\theta(g)=-2$
for any element $g$ of order $3$. Restricting $\chi$ to $L^+$, we arrive at a contradiction.

Now assume that $\ep=+$, and take $\chi \in \Irr(G)$ of degree $q^3+1$.  Then $\chi(u) =q(q+1)$, see \cite{CR} and \cite{EY2}. We use the notation of \cite{Ge}, and decompose
$$\chi|_{L^-} = a\chi_1+b\chi_{q^2-q}+c\chi_{q^3}+\sum^{q}_{i=1}(d_i\chi^{(i)}_{q^2-q+1}+e_i\chi^{(i)}_{q(q^2-q+1)})+
    \sum_{i,j}f_{ij}\chi^{(i,j)}_{(q-1)(q^2-q+1)}+\sum_jg_j\chi^{(j)}_{q^3+1},$$
where $a,b,c,...$ are non-negative integers. First, if some $g_j \geq 1$, then $\chi|_{L^-} =  \chi^{(j)}_{q^3+1}$, and hence $|\chi(u)| \leq 2$, a contradiction.
So $g_j=0$ for all $j$. Likewise, if $c \geq 1$, then $c=1$, $\chi|_{L^-} =  \chi_1+\chi_{q^3}$, yielding $\chi(u)=2$, again a contradiction.
Now, evaluating at an element of order $q^2-q+1$, we get $0=a-b$, i.e. $b=a$. Comparing the degrees, we obtain
$$q^3+1=(q^2-q+1)(a+\sum_id_i+\sum_iqe_i),$$
whence $a+\sum_id_i +\sum_iqe_i \leq q+1$. Now, evaluating at $u$, we get
$$q(q+1) = \chi(u) = 2a+\sum_i d_i + \sum_ie_i \leq q+1+a\leq 2(q+1),$$ again a contradiction.
\enp

If $r\neq 3$ then  the subgroups of $G$ isomorphic to $H_i$ for $i=1$ or $2$ are conjugate, as follows from the classification of maximal subgroups of $G$ obtained in \cite{Coo} and \cite{Kl2}.
The case with $r=3$ has features which force us to consider this separately (Lemma \ref{3g2}).

 \begin{lemma}\label{G2}
{\rm \cite[Lemma 4.10]{Z4}}  Let $p,q>2$, $p \nmid q$ and let  $g\in G\cong  G_2(q)$  be  a $p$-element  contained in a proper parabolic subgroup of $G$. Let
$1_G\neq \theta     \in  \Irr _\ell G$.  Then  $\deg\theta     (g)=|g|$.  \end{lemma}

\begin{remar} {\em In \cite[Lemma 4.10]{ Z4}  it is assumed that  $\ell\neq p$, however, this  is nowhere used in
the prove of \cite[Lemma 4.10]{Z4}; so the claim is true for $\ell=p$ as well.}
\end{remar}

\begin{lemma}\label{g2t}
Let $G=G_2(q)$, $q>2$, and let $\theta     \in\Irr _FG$    with $\dim\theta   >1$. Let $g\in G$ be a
semisimple $p$-element. Suppose that $g\in H= \SL_3(q) $.   Then  $\deg\theta    (g)=|g|$. \end{lemma}

\begin{proof}
Suppose the contrary. Let  $\phi$ be a non-trivial \ir constituent of $\theta|_H$.
If $g$  is contained in a parabolic subgroup of $H$,
then the result follows from Corollary \ref{hh2}, unless $ g\in Z(H)$, so $|g|=3$. Then
then it normalizes a nontrivial $r$-subgroup, whence
$g$ lies  in a parabolic subgroup of $G$, and the result follows from Lemma \ref{G2}.
So we assume that $g$ is not in a parabolic subgroup of $H$. Then $p>3$ and, by Lemma  \ref{ss3},  $|g|=q^2+q+1$. Then Sylow $p$-subgroups of $G$ are cyclic. If $\ell=p$ or $\ell=0$ then the result in this case follows from \cite{Z2}. Suppose that $\ell \neq 0,p$.
Then, again by Lemma \ref{ss3}, $\deg\phi(g)=\dim\phi=q^2+q$,  and 1 is not an \ei of $\phi(g)$.
If $1_H$ is a constituent of $\theta|_H$ then $\deg\theta(g)=|g|$. Otherwise, $\dim\theta$ is a multiple of
 $q^2+q$. Then  $\theta$  is liftable (this  follows by inspection of Brauer character degrees of $G$ available in \cite{Sh,Sh2} for $\ell|(q^2+q+1)$), and we are back to the complex case.
 \end{proof}

In view of Lemma \ref{g2t} to prove Theorem \ref{g2o} we can assume that $g\in H\cong \SU_3(q)$.
Moreover, by \cite[Lemma 6.1]{TZ8}, if $p>2$ then either $|g|=q+1$ or $|g|=q^2-q+1$.  In the former case $q$ is even. Let $p=2$. Then, by Lemma \ref{cv8}
and Proposition \ref{pr5}, either $|g|=q+1$ or $2(q+1)$.




 For $H=\SU_3(q)$ the arguments in the proof of Lemma \ref{g2t} do not work. Indeed, if  $p|(q+1)$ then we cannot use \cite{Z2} for  $p=\ell$ as Sylow $p$-subgroups of $G$ are not cyclic. So we turn to another method. We show that the restriction
 to  $H$ of every non-trivial \irr of $G$ contains a non-trivial \ir constituent which is not Weil.
 This will imply Theorem \ref{g2o} in view of Proposition \ref{pr5} for $p=2$ and \cite[Lemma 6.1]{TZ8}.
   In addition, to handle the case of ${}^3D_4(q)$ we need a similar result for $H=\SL_3(q)$ and $\ell=q^2+q+1$. In fact, for our use it suffices to consider the case where $Z(H)=1$.
  To deal with $\ell|(q^4+q^2+1)$ we first prove the result for $\ell=0$ and next use the decomposition numbers to deal with $\ell$-Brauer characters. In our reasoning below a certain role is played by the
Gelfand-Graev \rep of $H$. This is the induced \rep $\lam ^G$, where $\lam$ is a so-called non-degenerate
linear character of the maximal unipotent subgroup of $G$.

  \bl{r0r}
  Let $G=G_2(q)$, $q>2$, and $H\cong \SL_3^\ep(q)< G$, $(3,q-\ep)=1$. Let $1_G\neq \tau\in\Irr(G)$ (so $\ell=0$ here). Then $\tau|_H$ contains a non-trivial \ir constituent which is not Weil.\el

 \begin{proof} Let $U$ be a maximal unipotent subgroup of $H$ and let $\Gamma=\Gamma^\ep$ be the Gelfand-Graev character of $H$. Recall that $Z(H)=1$ implies that the   Gelfand-Graev \rep of $H$ is unique (this follows from \cite[14.28 and 14.29]{DM}).   It is well known that neither
  $1_H$ nor Weil are constituents of $\Gamma$. (Indeed, if $r|q$ is a prime  then the degrees of any \ir constituent
  of $\Gamma$  is $|C_H(s)|_{r}\cdot (|H|/|C_H(s)|)_{r'}$ for some semisimple element $s\in H$, see \cite[Theorem 8.4.9]{C}. One observes that 1, $q^2-\ep q$, $q^2-\ep q+1$ are not of these form.) Therefore, it suffices to show that
 \begin{equation}\label{eq1}
  (\tau|_H, \Gamma)>0
\end{equation}
 for every non-trivial \ir character $\tau$ of $G$.
  As the character table of $G$ is known, this can be easily checked. Indeed, observe that $\Gamma(x)=0$
  if $x\in H$ is not unipotent. In addition,
  $\Gamma(1)=|H|/|U|$ and $\Gamma(u)=1$ if $u\in H$ is regular unipotent \cite[14.45]{DM}. To compute $\Gamma(t)$ for $t\in H$ a transvection, one can use the inner product $(\Gamma,1_H)=0$. The conjugacy classes $u^H$, $t^H$ of $u$, $t$, are of size $|H|/q^2=q|H|/|U|$ and $|H|/|C_H(t)|=|H|/|U|(q-\ep)$, respectively. So we have
 $$\Gamma(t)\cdot|C_H(t)=|\Gamma(t)|H|/|U|(q-\ep)=\Gamma(u)q|H|/|U| -|H|/|U|,$$
 whence $\Gamma(t)/(q-\ep)=q-1$ and $\Gamma(t)=(q-\ep)(q-1)$. Therefore,
 $$|H|(\tau|_H, \Gamma)=\tau(1)\frac{|H|}{|U|}+\tau(t)\Gamma(t)|C_H(t)|+\tau(u)|C_H(u)|=\frac{|H|}{|U|} \bigl(\tau(1)+\tau(t)(q-1)+q\tau(u)\bigr),$$
 so $(\tau|_H, \Gamma)>0$ \ii $\tau(1)+\tau(t)(q-1)+q\tau(u)>0$. Inspection of the character table of $G$
 in \cite{CR} (for $(q,6)=1$), \cite{En1} (for $3|q$)  and \cite{EY2} (for $2|q$) yields the result. \enp

\bl{s32}
Let $H=\SL^\ep_3(q)$ 
and let $ \tau\in\Irr H$ with $\tau(1)>q^2+\ep q+1$. Then 
$\tau^\circ$ has an \ir constituent of degree greater than $q^2+\ep q+1$.
\el

  \bp 
  The result can be easily deduced from the comments in
  \cite[Examples 3.1 and 3.2]{TZ8}. \enp

\bl{r5r}
Let $G,H$ be as in Lemma {\rm \ref{r0r}}.
Let $1_G\neq \phi\in\Irr _\ell G$. Suppose that 
$\phi+m\cdot 1_G$ is liftable for some integer $m\geq 0$. Then $\phi|_H$ contains a non-trivial \ir constituent which is not Weil.\el

  \bp Let $\tau\in\Irr(G)$ be a character of $G$ such that   $\tau^\circ=\phi+m\cdot 1_G$.
 Then the non-trivial \ir Brauer characters of $H$ occurring in $\phi|_H$
 are the same as those in $(\tau|_H)^\circ$. By Lemma \ref{r0r}, there is an \ir constituent $\nu$
 in $\tau|_H$ such that $\nu(1)>q^2+\ep q+1$. By Lemma \ref{s32}, $\nu^\circ$ has an irreducible constituent of degree greater than $q^2+\ep q+1$. Whence the lemma. \enp

 \bl{r8r}
 Let $G,H$ be as in Lemma {\rm \ref{r0r}} and $1_G\neq \phi\in\Irr _\ell G$, where  $\ell|(q^4+q^2+1)$. Then $\phi|_H$ contains a non-trivial \ir constituent which is not Weil.\el

 \bp   (i) Recall that $\Gamma=\lam^H$ for the Gelfand-Graev character $\Gamma$ of $H$, where $\lam$ is a linear character of $U$.
 Therefore,  $\Gamma^\circ=(\lam^\circ)^H$, so $\Gamma^\circ$ is the Brauer character of a projective
 $FH$-module. In particular, if $\alpha$ is any class function on $H$, then the inner product $(\alpha^\circ,\Gamma^\circ)'$ over $\ell'$-elements of $H$ is equal to the usual inner product $(\alpha,\Gamma)$ over all elements in $H$. Also, as $\Gamma$ is \mult free,
 it follows that $\Gamma^\circ$ is a sum of Brauer characters of indecomposable projective modules $P_{\psi_i}$ of $H$, each of them occurs with \mult 1. Therefore, $(\phi|_H,\Gamma^\circ)'$   equals the sum of multiplicities of the \ir constituents $\nu$ of $\phi|_H$ such that $(\nu,\Gamma^\circ)'>0$.

 Note that $(1_H,\Gamma^\circ)'=(1_H,\Gamma)=0$ and hence $\nu\neq 1_H^\circ$. Let $\mu$ be a  \ir Weil character of $H$;
 in particular, $(\mu,\Gamma^\circ)' = (\mu,\Gamma) = 0$ as mentioned above. It is well known that
$\mu^\circ=\mu'+a\cdot 1_H$, for some $\mu' \in\Irr_\ell H$ and $a \geq 0$. Therefore,
$(\mu',\Gamma^\circ)'=(\mu,\Gamma^\circ)' = 0$, so $\mu'\neq \nu$. \itf $\phi|_H$
 has an \ir constituent that is neither trivial nor  Weil \ii $(\phi|_H,\Gamma^\circ)'>0$.

\smallskip
(ii) We will now prove that $(\phi|_H,\Gamma^\circ)' > 0$.
Denote by $\Irr^*_\ell G$ the set of characters $\phi\in\Irr _\ell G$ that are not of the form $\tau^\circ-a\cdot 1_G$ for some $\tau\in\Irr(G)$ and $a\geq 0$.
 Suppose first that  $\phi\notin \Irr^*_\ell G$, so that $\phi=\tau^\circ-a \cdot 1_G$ for some $\tau \in \Irr(G)$ and $a \geq 0$.
 Then $(\phi,\Gamma^\circ)'=((\tau|_H)^\circ,\Gamma^\circ)' = (\tau|_H,\Gamma)$, and \eqref{eq1} implies that $(\tau|_H,\Gamma)>0$.

Suppose that $\phi\in \Irr^*_\ell G$. As in Lemma \ref{r0r}, we observe that $(\phi,\Gamma^\circ)'>0$ \ii
\begin{equation}\label{eq2}
  \phi(1)+(q-1)\phi(t)+ q\phi(u)>0.
\end{equation}

To verify the inequality \eqref{eq2} we need the character values of $\phi$
 at  $u,t$, where $t\in H$ is a transvection and $u$ is a regular unipotent element of $H$.
 These can be computed from the decomposition numbers of $G$ modulo $\ell$, which are available in Shamash \cite{Sh,Sh2} for $3\neq \ell|(q^2+\ep q+1)$    and \cite{H3} for  $\ell|(q^2-1)$. %

 Suppose first that $3\neq \ell|(q^2+\ep q+1)$. 

  In notation of
 \cite{H3} the characters in $\Irr^*_\ell G$ that are in the principal block are $(X_{3}-X_{18})^\circ$ if $\ep=1$, and $ (X_{12}-X_{16}+1_G)^\circ$ if $\ep=-1$. 
From this one easily checks \eqref{eq2}.

 Suppose that $\phi$ is not in the principal block. Then either  $\phi$ is  liftable or
 $\ell|(q^2-\ep q+1)$ and $3|(q-\ep)$, in each case there is a single non-liftable character.
 If $\ep=1$ then the non-liftable character are
 $(X_{33}-X_{32})^\circ$ and 
 if $\ep=-1$ then this is $(X_{31}-X_{32})^\circ$, see \cite[p.190]{H3}. As above, one easily checks
 the above inequality \eqref{eq2}.

 For $(q,6)\neq1$ and  $\ell|(q^2+\ep q+1)$ the decomposition numbers are determined by Shamash \cite{Sh2}.
 The decomposition numbers for the characters in the principal block and, for $2|q$, in the other blocks containing non-liftable characters are the same as for $q$ with $(q,6)=1$. If $3|q$ then
 all non-liftable characters are in the principal block. The character tables of $G_2(q)$ with
 $(q,6)>0$ can be found in \cite{En1} for $3|q$  and \cite{EY2} for $2|q$. Inspection of
 \cite{EY2}  shows that  $X_i(1)$ and $X_i(t)$ for non-liftable characters $X_i$  are the same polynomials in $q$ as for $q$ with $(q,6)=1$. In addition, the absolute value of $X_i(u)$ is small enough
 to satisfy the inequality \eqref{eq2}.
\enp

We are left with the cases where $\ell|(q^2-1)$. As the character table of $G_2(q)$ with $3|q$ differs from
that for $q$ with $(q,3)=1$, we consider this case separately.

\bl{3g2}
Theorem {\rm \ref{mth1}} is true for $q=3^m$.\el

\bp By Lemma \ref{g2t} and the comments following it, we may assume that $p=2$ and $q+1$ is a 2-power.
We assume $|g|>2$ as the case $|g|=2$ is trivial.
Observe that  $|q+1|_2= 4$ if $m$ is odd and $|q+1|_2= 2$ otherwise.
Therefore $q=3$.
If  $\ell\neq 2$ then the result follows by the Brauer character table of $G_2(3)$ \cite{JLPW}.

Let $\ell=2$. Let $H\cong \SU_3(3)< G$.  We show that $\phi|_H$ has a non-trivial \ir constituent which is not Weil. As above, it suffices to check that $(\phi,\Gamma)>0$, where $\Gamma$ is the Gelfand-Graev character of $H$. This holds if $\phi(1)+(q-1)\phi(t)+q\phi(u)>0$. This can be easily checked using the Brauer character table of $G$ for $\ell=2$ in \cite{JLPW}.
\enp


From now on, until the end of this subsection we assume $r\neq 3$. To complete the proof of Theorem \ref{g2o}
it suffices to show that $\phi|_H$ has a non-trivial constituent which is not Weil (here $\phi\in\Irr_\ell G,$ $3\neq \ell|(q^2-1)$). We can do this by the method used in the proof of Lemma \ref{r8r}, however, there is a more conceptual approach to do this.

\bl{a12}
Let $G=G_2(q)$, $q>2$, $3 \nmid q$, and let $P$ be a long-root parabolic subgroup of $G$.
Let $U$ be the unipotent radical of $P$,
 $V$ be a non-trivial \ir $FG$-module and $V_0=C_V(Z(U))$.
Then $U/Z(U)$ acts on $V_0$  faithfully. \end{lemma}

\begin{proof} Assume the contrary: $U/Z(U)$ acts non-faithfully on $V_0$. It is well known that for $q>2$, $(q,3)=1$ the group $U/Z(U)$ has no non-trivial proper  $P$-invariant subgroup.
(See \cite{ABS} or \cite[Theorem 17.6]{MT} for  $(q,6)=1$ and \cite{Ko} for $q$ even; note that this fails for the excluded case $q=2$). As $V_0$ is a $FP$-module, it follows that $U$ acts trivially on $V_0$, i.e. $V_0=C_V(U)$ and thus $V = [V,Z(U)] \oplus V_0$.

It is also well known that $U$ is generated by the root subgroups $U_{\beta}$, $U_{\al+\beta}$, $U_{2\al+\beta}$, $U_{3\al+\beta}$, $U_{3\al+2\beta}$ and $[U,U]=U_{3\al+2\beta}=[U_{\beta},U_{3\al+\beta}]= Z(U)$.
Moreover, for every $1\neq y\in (U\setminus Z(U))$ we have $[U,y]= Z(U)$. It follows that the character of every \irr of $U$ non-trivial on $Z(U)$ vanishes on $y$.

Note that there exists some $y\in  (U\setminus Z(U))$ so that the $G$-conjugacy class of $y$ meets $Z(U)$, say at a (long-root) element $x$; indeed, both $U_{3\al+\beta}$, $U_{3\al+2\beta}$ are long root subgroups of $G$, and hence $G$-conjugate. Let $\varphi_0$ and $\varphi_1$ denote the Brauer characters of $V_0$ and $[V,Z(U)]$. Then $\varphi_0(x)=\varphi_0(y) = \dim V_0$,
and  $\varphi_1(y) = 0$. On the other hand, all elements of $Z(U)$ are conjugate in $G$, so all non-trivial \ir \reps of $Z(U)$ occur in $V_2|_{Z(U)}$ with the same multiplicity, whence
$\varphi_1(x)=-\varphi_1(1)/(q-1) < 0$. Thus $(\varphi_0+\varphi_1)(x) \neq (\varphi_0+\varphi_1)(y)$, a contradiction.\end{proof}

\begin{lemma}\label{ti1}
Let  $G = G_2(q)$, $q>2$, $3 \nmid q$, and let $V$ be a non-trivial irreducible $F G$-module.
Let $H \cong \SU_3(q)< G$. Then the restriction of $V$ to $H$
 has a non-trivial composition factor which is not an irreducible Weil module.\end{lemma}

\begin{proof}
Recall that $G$ has a single conjugacy class of subgroups isomorphic to $\SU_3(q)$ if $3 \nmid q$, see \cite{Kl2} and \cite{Coo}.

We use notation of Lemma \ref{a12}. Let $r$ be a prime dividing $q$.
Let $P$ be a parabolic subgroup specified in Lemma  \ref{a12} and $U$ the unipotent radical of $P$.

Let  $R \in {\rm Syl}_r (H)$; then $R'=Z(R)$. Observe first that the elements of $R'$ are conjugate with those in $Z(U)$. This follows from a result of Kantor \cite[p. 377]{Ka}, stated that
$H$ is generated by some subgroups of $G$ conjugate to $Z(U)=U_{3\al+2\beta}(q)$. Indeed, if $r\neq 2$ then
in a 7-dimensional \rep $\rho$, say, of $G$ over $\overline{\FF}_r$
the elements of $Z(U)$ satisfy the equation $(x-\Id)^2=0$; this is the case only for elements of $R'$ in $R$
as $\rho(M)$ is a direct sum of two \reps of degree 3 and the trivial one  \cite{Ka}. If $r=2$ then the elements of $Z(R)$ are involutions, whereas all involutions of $R$ lie in $R'$. \itf a subgroup of $H$ conjugate with $Z(U)$ is conjugate in $H$ to $Z(R)$.

Therefore, we may assume that $R'=Z(U)$. Let $B=N_H(R')$, so that $B$ is a Borel subgroup of $H$, and $|B|=q^3(q^2-1)$. One observes that $B$, acting on $R/R'$ by conjugation, either permutes transitively the non-identity elements of $R/R'$, or has $3$ orbits of the same length $(q^2-1)/3$ (in the latter case $Z(H)\cong C_3$). In particular, the only $B$-invariant subgroups of $R$ that contain $R'$ are $R'$ or $R$.
Now, as $U = O_r(P)$ and $B = P \cap H$, we have that $U \cap B \leq O_r(B) = R$ and $U \cap B \lhd B$. If moreover
$U \cap B = R'$, then $|BU/U| = |B/R'|$ has order $q^2(q^2-1)$, where Sylow $r$-subgroups of $P/U$ have order $q$, a contradiction. Hence $U \cap B = R$.

Suppose the contrary, that every \ir constituent of $V|_H$ is either trivial module or a Weil module.
By \cite[Lemma 11.1]{GMST}, for a Weil module $L$, say, the restriction $L|_R$ contains no
{\it nontrivial}  linear character of $R$. So the same is true for $V|_R$ and thus $R$ acts trivially on $C_V(Z(U))$. However, this contradicts Lemma \ref{a12} as $Z(U)< R< U$. \end{proof}

   \def\elts{elements }

Even though $G=G_2(2)$ is not simple, we still consider it for completeness. Note that it contains a normal subgroup $H\cong \SU_3(3)$ of index 2.

   \bl{gg2}
   Let $G=G_2(2)$ and let $g\in G$ be a  p-element for an odd  prime p.
   Let $\phi\in\Irr_\ell G$ for $\ell>2$. Suppose that $\deg\phi(g)<|g|$. Then $|g|\in \{3,7\}$, $\dim\phi=6$ and $\deg\phi(g)=|g|-1$.\el

\bp Note that $|G|=2^6\cdot 3^3\cdot 7$ so  $p=7$ or 3.
 If $p\neq \ell $ then the result follows by inspection of the Brauer character tables. Let $p=\ell$.

 If $p=7$ then Sylow 7-subgroups of $G$ are cyclic. By \cite[Lemma 3.3(v)]{Z2}, if $\deg\phi(g)<|g|$ then $\dim \phi\leq 6$. As $g\in H\cong\SU_3(3)$, the result follows from the main theorem of \cite{Z2}.

Let  $|g|=3$ and $\ell=3$.  If $\dim\phi=6$ then $\phi|_H$ is a direct sum of two \ir \reps of degree 3, which are Galois conjugate to each other. \itf  the Jordan form of $\phi(g)$
is $\diag(J_2,J_2, 1,1)$. Suppose that  $\dim\phi>6$. By Clifford's theorem, $\phi|_H$ is either \ir or a direct sum of two \ir \reps of equal degrees. \itf there exists a 3-modular \irr $\tau$ of $H$ of
degree $d>3$ such that $\deg \tau(g)=2$. This contradicts a result of \cite[Theorem 1]{PS}.\enp

 \bl{2g24}
 Let $G=2\cdot G_2(4)$, $1_G\neq \phi\in\Irr_\ell G $ for $\ell\neq 2$ and let $g\in G$ be a $p$-element of $G$ for $p>2$. Suppose that $\deg\phi(g)<|g|$. Then $\dim\phi=12$, $|g|\in\{3,7,13\}$ and $\deg\phi(g)=|g|-1$.
\el

\bp If $\ell\neq p $ then the result follows from the Brauer character table of $G$ \cite{JLPW}.
Let $\ell=p$. If  $|g|\in \{7,13\}$ then Sylow   $ p$-subgroups of  $G$  are cyclic and the result is contained in   \cite{Z2}. We are left with $p\in\{3,5\}$ and $|g|=p$. Let $|g|=3$. If $g\in 3B$ then $g$ is contained in a subgroup $X\cong \PSL_2(13)$, so $\deg\phi(g)=3$ for every $1_G\neq\phi\in\Irr_3G$ as this holds for  $X$.
 Suppose that  $g\in 3A$. As $g$ has just two distinct \eis
in an \irr $\tau$ of $G$ of degree 12 over $\CC$,  the minimum \po degree of $g$ in the reduction of this modulo 3 equals 2 as well.

Suppose that $\dim\phi>12.$ Note that $G/Z(G)$ contains a subgroup $Y\cong \SU_3(3)=G_2(2)'$. As the Schur multiplier of $Y$ is trivial, we can assume that $Y$ is a subgroup of $G$. As $|Y|_3=|G|_3=27$, we
observe that $g$ is conjugate to an element of $Y$. We claim that the non-trivial composition factors of $\phi|_Y$ are of dimension 3. Indeed, if $\lam\in\Irr_3(Y)$ is such a factor then $\deg\lam(g)=2$. Then $\lam$
extends to a \rep $\overline{\lam}$ of $\SL_3(\overline{\FF}_3)$. By \cite{PS}, $\overline{\lam}$
is a Frobenius twist of an \irr with highest weight $\om_1$ or $\om_2$, which are the fundamental weights of
the weight system of $\SL_3(\overline{\FF}_3)$. These \reps are of dimension 3, whence the claim.
Note that $Y$ itself has two \ir \rep over   $\overline{\FF}_3$ dual to each other, and we denote their Brauer characters by $\lam_1,\lam_2$.

Let $h\in G$ be of order 7. There is a single conjugacy class of such elements, so $h$ is rational
and hence $\be(h)$ is an integer, where $\be$ is the Brauer character of $\phi$. \itf
$\be|_Y=a(\lam_1+\lam_2)+b\cdot 1_Y$, so  $\be(1)=6a+b$. Therefore, $\be(h)=-a+b$, whence
$a=(\be(1)-\be(h))/7$, $ b=a+\be(h)$.

In view of Theorem \ref{g2o}, $\phi$ is faithful. 
Furthermore, $Y$ has a unique involution $z$, say, and $\lam_1(z)=\lam_2(z)=-1.$ Therefore, $\be(z)=-2a+b$.
We use the Brauer character table of $G$ for $\ell=3$ in \cite[p. 274]{JLPW}.  Note that $z$ is in class $2A$ in $G/Z(G)$, as $\be(z)=-4$ if $\be(1)=12$. By \cite[p. 274]{JLPW}, we have 

 \med
\begin{center}
\begin{tabular}{|c|c|c|c|c|c|c|c|c|c|}\hline
$\be(1)$&$12$&$104$&$ 352$ &$1260$ &$ 1364$ &$1800$&$ 2016$&$3744$&$ 3888$ \cr
\hline
$\be(h)$&$-2$ &$- 1$&2&0& $-1$&1&$0$&$-1$&3 \cr \hline
$\be(z)$&$- 4$&$8$ &$-32$&$-36$ &$-28$ &40&$96$&$32$&$-16$\cr \hline
$\be(1)-\be(h)$& $ 14$& 105&350&$1260$ &1365 &$1799$&2016&3745&$3885$ \cr \hline
$a$& $2$&15&50 &$180$ &195 &257&288&$535$&$555$ \cr
\hline
$b$& $0$&14&52 &$180$ &194 &258&288&$534$&$558$ \cr
\hline
\end{tabular}
\end{center}

\noindent It is obvious that the relation $\be(z)=-2a+b$ holds only for $\be(1)=12$.

\med
Let $|g|=5$. Note that a subgroup $H\cong \SU_3(4)$ contains a Sylow 5-subgroup of $G$.
The \irr of $G$ of degree 12 (in any characteristic) remains \ir under restriction to $H$, and this is a Weil \rep of $H$.
By \cite[Lemma 6.1]{TZ8}, $\deg\phi(g)=4$ if $\phi$ is a Weil \rep of degree 12.

Let $\phi\in\Irr_5(G)$ and $\phi(1)>12.$ Then $\phi(1)\neq 1800,3600,3900$ as these characters are of 5-defect 0. In notation of \cite{JLPW}, we are to inspect the characters $\phi_i$ with $i=22,23,24,25,28,29,30$.

By \cite[Lemma 6.1]{TZ8}, every non-trivial \ir constituent of $\phi|_H$ is of degree 12. So  $\phi|_H=a\cdot \tau+b\cdot 1_G$, where $\tau\in\Irr_5(H)$ with $\tau(1)=12$.
Then $\phi(1)=12a+b$. Let $g_{13}\in H$ be of order 13. Then  $\tau(g_{13})=-1$, so $\phi(g_{13})=b-a$.
In particular, $\phi(g_{13})$ is an integer, and the Brauer character table for $\ell=5$ shows that
$\phi(g_{13})\in\{0,1,-1\}$. Let $g_2$ be of order 2 then $\tau(g_2)=-4$, whence  $\phi(g_2)=-4a+b$.

If $a=b$ so $\phi(g_{13})=0$ and $\phi(1)=13a$, so $13|\phi(1)$, whence $i=23,28,30$. In addition, $\phi(g_2)=-3a$ so $3|\phi(1)$, which is false for $i=23,28,30$.

Let $b=a+1$. Then $\phi(g_{13})=1$, 
whence $i=22,24,29$. As $\phi(1)=13a+1$, we have $a=7,43,167$, resp. In addition, $\phi(g_2)=-3a+1=-20,-128,500$, which is false.

Let $b=a-1$. Then $\phi(g_{13})=-1$,  whence $i=25$. Then $\phi(1)=1260=13a-1$, whence $a=97$. Then
$\phi(g_2)=-3a-1=-592$, which is false again.
\enp

 \section{The case of ${}^3D_4(q)$}

In this section we consider the groups $G={}^3D_4(q)$ and prove the following result.

\begin{theo}\label{td43}
Theorem {\rm \ref{mth1}} is true for groups of type ${}^3D_4(q)$.
\end{theo}

We first consider the case where Sylow $p$-subgroups are cyclic and next the remaining cases.

\subsection{The case of cyclic Sylow $p$-subgroups} 

 Note that $G$ contains a cyclic torus of order  $q^4-q^2+1$.
As $|T|$ is coprime to $|G|/|T|$, the Sylow  $p$-subgroups of $T$ and of $G$ are cyclic
for $p$ dividing $|T|$, and $C_{G}(t)=T$ for every $1\neq t\in T$, that is, $t$ is regular. So
the assumptions of Lemmas \ref{sy1} and \ref{cs2} hold. Therefore, if $\phi\in\Irr_\ell G$ then
either $\phi|_T=k\cdot \rho_T^{\reg}$ with $k>0$ or $\phi$ is liftable or $\phi$ is unipotent.
If $\ell\in\{0,p\}$ then Theorem \ref{td43} follows from \cite{Z2}, so we assume $\ell\neq 0,p$.
In addition, we can assume that $\phi$ is not liftable. By Lemma   \ref{cs2}, we are left with
$\ell$-modular Brauer \ir characters from the unipotent blocks, that is, we assume that $\phi$
is a constituent of a unipotent character modulo $\ell$.

We first specify Lemma \ref{tr1} for our situation. Note that the degree of any non-trivial
Brauer character of $G$ (provided $\ell$ is coprime to $q$) is at least $q^5-q^3$ \cite[Table 1]{SZ}.

\bl{ct1}\mar{ct1} In the notation of Lemma {\rm \ref{tr1}}, let $C=T< G$ so that $|C|=q^4-q^2+1$. Let $1_G\neq \phi\in\Irr_\ell G$. Suppose that $\phi(g)=c<0$ for all $g\in C$, and that $-c< q$. Then   $\phi|_{C}-\rho_C^{\reg}$ is a proper character of $C$.
\el

\begin{proof} We have $-c(|C|-1)<q^5-q^3 \leq \phi(1)$, so the result follows from Lemma \ref{tr1}.\end{proof}

The $\ell$-decomposition matrix of $G$   is determined by Geck \cite{Ge2} for $\ell>2$ and
Himstedt \cite{Hi2} for $\ell=2$, but a few  entries for which only partial information
has been obtained. For $\ell>3$ Dudas \cite{Du1} has determined some of those entries.
For undetermined entries we need  upper bounds; for $\ell=2$  these are available from  \cite{Hi2},
and for $\ell>2$ these can be read off from the proof given in \cite{Ge2}. We acknowledge  Dr. Himstedt's
  help with this matter.

There are 8 unipotent characters of $G$, denoted by ${\bf 1}, [\ep_1],[\ep_2],[\rho_1],[\rho_2], \StC, {}^3D_4[1]$, ${}^3D_4[-1]$ in \cite{Ge2} and elsewhere. We simplify this notation below by setting $D^+={}^3D_4[1]$,  $D^-={}^3D_4[-1]$ and using $1_G,\ep_1,\ep_2,\rho_1,\rho_2$ in place of
 ${\bf 1}, [\ep_1],[\ep_2],[\rho_1],[\rho_2]$.

We have $\ep_1(1)=q(q^4-q^2+1)$, $ \ep_2(1)=q^7(q^4-q^2+1)$, $ \StC(1)=q^{12}$, $ \rho_1(1)=q^3(q^3+1)^2/2$, $  \rho_2(1)=q^3(q+1)(q^4-q^2+1)/2$, $ D^+(1)=q^3(q-1)^2(q^4-q^2+1)/2$, $D^-(1)=q^3(q^3-1)^2/2$.

Furthermore $ \ep_1(1)\equiv 0\pmod{T}$, $\ep_2(1)\equiv 0\pmod{T}$, $\rho_2(1)\equiv 0\pmod{T}$,  $ D^+(1) \equiv 0\pmod{T}$, $\rho_1(1)\equiv -1\pmod{T}$, $ \StC(1)\equiv 1\pmod{T}$, $D^ -(1)\equiv 1\pmod{T}$. This implies
$ \ep_1(t)=  \ep_2(t)=\rho_2(t)= D^+(t)= 0$, $\rho_1(t)=-1$, $ \StC(t) =D^ -(t)=1$ for $1\neq t\in T$.

Himstedt \cite{Hi2} identifies 
the $\ell$-modular \ir \reps of $G$ of degree $<(q^5-q^3+q-1)^2$. As a consequence of this, we have

\bl{hr1}
Let $1_G\neq \phi\in\Irr_\ell G$. Then either $\phi(1)>q^8+q^4$ or
$\phi$ is an \ir constituent of $\ep_1^\circ$ and $\phi(1)=\ep_1(1)$ or $\ep_1(1)-1$.\el

\med
Let $T$ be a torus of order $q^4-q^2+1$ and $1\neq t\in T$ is an arbitrary $\ell'$-element. Using the data from \cite{Ge2} and \cite{Hi2}, we
will show that either $\phi(t)\geq 0$ or $-\phi(t)\cdot |T|<\phi(1)$ whenever $\phi\neq 1_G$ is a unipotent Brauer character of $G$. For this we first obtain an upper bound for $\phi(t)$.

 Recall that $\Irr^0_\ell(G)$ denotes
the set of non-liftable unipotent Brauer characters of $G$, and use $\phi_i$ with $1\leq i\leq  |\Irr^0_\ell(G)|$ to denote the Brauer characters in $\Irr^0_\ell(G)$. (This notation for $\phi_i$ does not coincide with the one used in \cite{Hi2}.)


\med
If $2,3\neq \ell |(q-1)$ then $|\Irr^0_\ell(G)|=0$, so every \ir $\ell$-modular character is liftable.


\med
Let $\ell|(q^4-q^2+1)$. By \cite[p. 3265]{Ge2}, $|\Irr^0_\ell(G)|=2$,
and $\phi_1=\rho_1-1_G$,  $\phi_2=\StC-\rho_{1}$.
So $\phi_1(t)=-2$ and $\phi_2(t)=2$.

\med
Let  $2,3\neq \ell |(q+1)$. Then $|\Irr^0_\ell(G)|=3$ and $\phi_1=\ep_1-1_G$, $\phi_2=\ep_2- 1_G$,
$\phi_3=\StC-  \ep_1-\ep_2-aD^- -b D^+$, where $1\leq a,b\leq (q-1)/2$. (In fact, $a=b=2$ unless possibly $\ell=5$ and $q+1$ is not a multiple of 25 \cite[Theorem 2.3]{Du1}).
 Then $\phi_1(t)=\phi_2(t)=-1$, $\phi_3(t)=1-a$ so $-\phi_3(t)<q$. 

\med
Let  $3\neq \ell |(q^2+q+1)$. Then $|\Irr^0_\ell(G)|=4$ and $\phi_1=\rho_1-\ep_1$,
$\phi_2=\rho_2-1_G$, $\phi_3=\ep_2-\rho_1+\ep_1-a D^+ $,  $\phi_4=\StC -c\phi_3-b D^+ -\phi_2=
\StC-c\ep_2+c\rho_1-c\ep_1-acD^+-\rho_2+1_G-b D^+=\StC-c\ep_2+c\rho_1-c\ep_1-(b+ac)D^+-\rho_2+1_G$.

So $\phi_1(t)=\phi_2(t)=-1$, $\phi_3(t)=1$,
$\phi_4(t)=2-c$. In fact, $c=2$ by \cite[Theorem 2.4]{Du1}, so $\phi_4(t)=0$.

\med
Let  $3\neq \ell |(q^2-q+1)$. Then $|\Irr^0_\ell(G)|=3$ and $\phi_1=\rho_2-\ep_1-1_G$,
 $\phi_2=\ep_2-a D^- -\phi_1-1_G =\ep_2-a D^- -\rho_2+\ep_1$,
  $\phi_3=\StC -d\phi_2 -cD^+  -b D^- -\phi_1- \ep_1=
     \StC -d(\ep_2-aD^- +\ep_1+1_G) -cD^+  -bD^- -(\ep_1-1_G)- \ep_1=\StC -d\phi_2 -cD^+  -b D^- \phi_1- \ep_1=\StC -d\ep_2+(ad-b)D^- +d\ep_1+(1-d).1_G -cD^+$.

  So $\phi_1(t)=-1,$ $\phi_2(t)=-a$, $\phi_3(t)=2+ad-b $. Here $a=d=0$, $b=2$ by \cite[Theorem 2.5]{Du1},
  so $\phi_3(t)=0$.

\med
Let $\ell=3|(q-1)$. Then $|\Irr^0_\ell(G)|=4$ and $\phi_1=\rho_1-\ep_1$, $\phi_2=\rho_2-1_G$,
 $\phi_3=\ep_2-\phi_1-aD^+=\ep_2-\rho_2+\ep_1-a D^+ $,  $\phi_4=\StC -c\phi_3-b D^+ -\phi_2=
 \StC -c(\ep_2-\rho_2+\ep_1-a D^+)-b D^+ -\rho_2+1_G=\StC -c(\ep_1+\ep_2)+(c-1)\rho_2+(ac-b)D^+ +1_G$.

 So $\phi_1(t)=\phi_2(t)=-1$, $\phi_3(t)=1$, $\phi_4(t)=2-c$. Here $c\leq q$, so $-\phi_4(t)<q$.


 \med
Let $\ell=3|(q+1)$. Then $|\Irr^0_\ell(G)|=4$ and $\phi_1=\ep_1-1_G$, $\phi_2=\rho_2-\phi_1-2.1_G$, 
 $\phi_3=\ep_2-\phi_2-aD^- -1_G$,
   $\phi_4=\StC -d\phi_3 -c D^+  -b D^- -\phi_2-\phi_1+1_G=\StC -d\phi_3 -c D^+  -b D^--\rho_2(1)$, as $\phi_1+\phi_2=\rho_2-2.1_G$.


 So $\phi_1(t)=\phi_2(t)=-1$, $\phi_3(t)=-a$, $\phi_4(t)=2+ad-b$.

In this case $0\leq a\leq 1$, $a+1\leq b\leq 3(q+1)/2$, $c\leq (q-1)/2$ and $1\leq d\leq q$ \cite{Ge2}. 
So $-\phi_4(t)\leq \frac{3q-1}{2} $.

\med

Let $\ell=2$. Then by \cite[Theorem 3.1, p.572]{Hi2}, we have
 $|\Irr^0_\ell(G)|=5$ and $\phi_1=\ep_1-  1_G$, $\phi_2=\rho_1-D^+ $, $\phi_3=\rho_2-\phi_2 $,  $\phi_4=\ep_2-  1_G$, $\phi_5=\StC-\phi_4-aD^+ -b\phi_3- \phi_1 -1_G$,
 where $0\leq a,b\leq q$.
 We have $\phi_1(t)=\phi_2(t)=\phi_4(t)=-1$, $\phi_3(t)=1$, $\phi_5(t) =2-b$.

\bl{ub1}
Let $1_G\neq \phi\in\Irr^0_\ell G$ and T be a torus of order $q^4-q^2+1$ of G. Then either $\phi(t)>0$ or $-\phi(t)\cdot |T|<\phi(1)$ for every $1\neq t\in T$. In particular, $\deg \phi(t) = |t|$ for any $p$-element $t \in T$ with $p \neq \ell$.
\el

\bp Suppose first that $\phi$ is a non-trivial \ir constituent of $\ep_1$. Then $\phi(1)= \ep_1(1)-1$
(and either $\ell=2$ or $\ell=3|(q+1)$). Then $\phi(t)=-1$ and $\phi(1)=q(q^4-q^2+1)-1$, whence the claim.

Suppose that $\phi$ is not a constituent of $\ep_1$. Then $\phi(1)>q^8+q^4$ by Lemma \ref{hr1}, and either $\phi(t)>0$ or $-\phi(t)\leq (3q-1)/2$. Then  $|T|\cdot(3q-1)/2=(q^4-q^2+1) (3q-1)/2<q^8+q^4$.

For the last claim, observe by Lemma \ref{cs2} that $\phi$ takes a constant value $c$ on $\langle t \rangle \setminus \{1\}$. Now apply Lemma \ref{tr1}.
\enp

\begin{remar}
{\em In \cite{Hi2} there are weaker bounds for $a,b,c,d$ for $3=\ell|(q+1)$, specifically
$a\leq q(q-1)$, $b\leq (q^3-1)/2$, $c,d\leq(q-1)/2$. These are sufficient for our purpose, as either
$\phi(t)>1$ or $-\phi_4(t)\leq ((q^3-1)/2)-2$,
and again we have $q^3|T|<\phi_4(1)$.}
\end{remar}

\subsection{The case where Sylow $p$-subgroups are not cyclic}

For uniformity we denote by $\SO_{2n}(\overline{\FF}_q)$ the subgroup of index  $2$ of $O_{2n}(\overline{\FF}_q)$ if $q$ is odd and of $O_{2n}(\overline{\FF}_q)$ if $q$ is even.
Then $\SO_{2n}(\overline{\FF}_q)$ is a connected simple algebraic group of type $D_n$
If $q$ is even then $SO_{2n}(\overline{\FF}_q)$ is formed by elements of quasi-determinant $1$ in $O_{2n}(\overline{\FF}_q)$
Also, let $\GC=Spin_{2n}(\overline{\FF}_q)$ denote the simply connected simple algebraic group of type $D_n$, so
that $\GC=\SO_{2n}(\overline{\FF}_q)$ when $2|q$ and $\GC/C_2 = \SO_{2n}(\overline{\FF}_q)$ when $2 \nmid q$.
Taking $n=4$, we can view $G = \GC^F$ for some Steinberg endomorphism $F:\GC\to\GC$.
By \cite[p. 33]{Kl}, $G\cong {}^3D_4(q)$ has maximal subgroups isomorphic to $G_2(q)$ and $(C_{q^2+q+1} \circ \SL_3(q)).\gcd(3,q-1)$ (where $C_{q^2+q+1}$ is cyclic of order
$q^2+q+1$).
If $g\in G$ lies in the $G_2(q)$-subgroup then we can use our result on $G_2(q)$ (Theorem \ref{g2o}). So our first goal is to establish Lemma \ref{fr1} below.

   \bl{l33} Let $\mathbf{G}$ be a simple, simply connected algebraic group of type $D_4$ in defining characteristic $r$, and let $V$ be the standard  $\mathbf{G}$-module. Let $G\cong {}^3D_4(q)
< \mathbf{G}$, where $r|q$ and $q>2$.  Let $M_1\cong (C_{q^2+\ep q+1} \circ \SL^\ep_3(q)).\gcd(3,q-\ep)$ and $M_2\cong G_2(q)$ be maximal subgroups of G and $H_1 < M_1$ 
  a subgroup isomorphic to  $\SL^\ep_3(q)$.   
Then $H_1$ is   $\mathbf{G}$-conjugate to a subgroup of $M_2$.
\el

\bp Note that $V$ is self-dual. Our strategy is to show that for some subgroup $H_2\cong \SL^\ep_3(q)$ of $M_2=G_2(q)$
for every $i=1,2$ there are $H_i$-stable subspaces  $V^i_1,V^i_2,V^i_3$
such that $V=V^i_1 \oplus V^i_2 \oplus V^i_3$, $V^i_1,V^i_2$ are totally singular  of dimension 3,  and $V^i_3$ is non-degenerate and trivial on $H_i$. Then $V^i_1,V^i_2$ are dual $FH_i$-modules, and the result will follow from Witt's theorem. Indeed, by Witt's theorem applied to $\GO(V)$, there is some $x \in \GO(V)$ that sends $V^1_j$ to $V^2_j$ for $j = 1,2,3$.
In our case $V^i_3$ is a trivial $H_i$-module of dimension 2 and so we can find an element
$y \in (\GO(V^2_3) \setminus \SO(V^2_3))$ that commutes with $H_2$. Replacing $x$ by $yx$ if necessary, we may assume
$x \in \SO(V)$ and sends $V^1_j$ to $V^2_j$. Since all $3$-dimensional nontrivial representations of $H_i$ are irreducible and
quasi-equivalent to the natural representation, we now have that $x$ sends the image of $H_1$ in $\SO(V)$ to the image of
$H_2$ in $\SO(V)$, and we are done if $r=2$. If $r>2$, then an inverse image of $x$ conjugates
the full inverse image $C_2 \times H_1$ to $C_2 \times H_2$, hence $H_1$ to $H_2$.

If $r\neq 3$ then the subgroups of $M_2$ isomorphic to $\SL^\ep_3(q)$ are conjugate so choose for $H_2$ any of them. If $r=3$ then $G_2$ has two conjugacy classes of subgroups isomorphic to $\SL^\ep_3(q)$, one of which is reducible on the \ir $FM_2$-modules of dimension 7, and the other is irreducible \cite[Theorem A]{Kl}. In this case we choose $H_2$ from the former one.

It suffices to show that the composition factors of $V|_{H_i}$, $i=1,2$ are of dimension 1 or 3.
Indeed, as $V$ is self-dual and the simple $H_i$-modules of dimension $3$ are not self-dual,
there are two composition factors of dimension  3 and they are dual to each other. Let $N$ be a composition factor of dimension  3. It suffices to observe that ${\rm Ext}^1_{H_i}(N,N^*)=0$ and ${\rm Ext}_{H_i}^1(N,N_0)=0={\rm Ext}_{H_i}^1(N_0,N)$,
where $N_0$ is the trivial $FH_i$-module and $N'$ is dual to $N$.
 If $r>2$ then every $FH$-module of dimension almost 6 is completely reducible by
  \cite[Theorem 1.1]{Mc}. If $r=2$ then this follows by \cite{Si}.

Suppose first that $i=2$.  As the dimensions at most $8$ of nontrivial simple $G_2(q)$-modules are $6$ if $r=2$ and $7$ if $r \neq 2$ \cite[p. 167]{Lu}, the fixed point subspace $L$ of $M_2$ on $V$ is non-zero. Next, if  $r\neq 2$ then $K$ must be non-degenerate, and $V=L\oplus V'$, where $V'$ is an \ir $FM_2$-module. If $r=2$ then as $\dim V/L = 1$, $V/L$ is reducible, and has a composition factor $V'_1$ of degree 6. Note that the \ir constituents of the restrictions of these modules
$V'$, respectively $V'_1$ to $H_2$ are  3-dimensional and dual to each other.
(Indeed, if $r=2$ then all nontrivial simple $H_2$-modules of dimension $\leq 6$ are of dimension $3$ and non-self-dual, see \cite[p. 149]{Lu}. If $r \neq 2$, then by the choice of $H_2$, $H_2$ is reducible on $V'$. If moreover it has a composition
 factor $W$ of dimension $\neq 1,3$ on $V'$, then $\dim W = 6$ by \cite[p. 149]{Lu} and $W$ is not self-dual. The other composition factors of the $H_2$-modules are all trivial, and this contradicts $V \cong V^*$.)

Suppose  that $i=1$. Let  $K$ be an \ir $FH_1$-submodule of $V$ of maximal
dimension. Then $K\neq V$ by Schur's lemma. Suppose that $\dim K=7$. By \cite[p. 149]{Lu}, we have $r=3$ and
$V$ has a composition factor $K'$ of dimension 1. As $V$ is self-dual, we conclude that $V$ is completely reducible over $H_1$, so $K$ is non-degenerate,
$V=K\oplus K^\perp$ and  $\dim K^\perp=1$. Now the odd-order subgroup $C_{q^2+\ep q+1}$ must act trivially on both $K$ and $K^\perp$, a contradiction. Suppose that $\dim K=6$.  Then $r > 2$ and $K$
 is not self-dual \cite[p.149]{Lu}, so $V_{H_1}$ has a composition factor dual to $K$, which is not the case by dimension reason. So $\dim K\leq 5$ and hence $\dim K = 3$ by \cite[p. 149]{Lu}.
\enp

 \bl{sc1}
 Let $H\cong \SL^\ep _3(q)\subset \GC=\mathrm{Spin}_8(\overline{\FF}_q)$ and $V$ the natural module for $\SO_8(\overline{\FF}_q)$. Suppose that $V|_H = V_1 \oplus V_2 \oplus V_3$ with $V_1,V_2$ totally singular of dimension $3$ and $V_3$ a non-degenerate subspace on which $H$ acts trivially. Then $Y:=C_\GC(H)$ is connected. \el

\bp By assumption, $V_1,V_2$ are dual to each other and $V_3$ is trivial as $H$-modules. Hence $Y$ stabilizes these modules; let $\bar Y$ denote the image of $Y$ in $\SO(V)$. Then
$\bar Y \leq \{\diag(y_1,y_2,t)\}$, where $y_1\in \GL_3(\overline{\FF}_q)$ is a scalar $(3\times 3)$-matrix,
$y_2=y_1^{-1}$ and $t\in \GO(V_3)$.  In fact, $t\in \SO(V_3)$. This is obvious if $q$ is odd; if
$q$ is even then, since $\GL_3(\overline{\FF}_q)$ has no subgroup of index $2$, so the matrix $\diag(y_1,y_1\up,\Id_2)$ lies in $\SO(V)$,
whence $\diag(\Id_6,t) \in \SO(V)$, which implies the claim. 
Note that $C_{\SO(V)}(H)$ contains the subgroup $\diag(\Id_3,\Id_3,\SO(V_3))$. Hence $C_{\SO(V)}(H)$ is isomorphic to the direct product of $Z(\GL_3(\overline{\FF}_q))\cong \overline{\FF}_q^\times$ and $\SO_2(\overline{\FF}_q)$, which is a connected group.
Now, if $r=2$ then $Y = \bar Y = C_{\SO(V)}(H)$, and we are done. If $r > 2$, then these two subgroups lift to a one-dimensional torus $T$ and $S=\Spin_2$, respectively, which centralize each other modulo $C_2 = Z(\Spin(V))$. Since $S$ is perfect, we have that $[T,S] = [T,[S,S]]$ is contained in $[[T,S],S] = 1$, so the full inverse image of $C_{\SO(V)}(H)$ is a central product of two
connected subgroups, and so is connected. Each of $T$ and $S$ centralizes the perfect subgroup $H$ modulo $Z(\Spin(V))$,
so the same argument shows that $Y$ is the full inverse image of $C_{\SO(V)}(H)$, and so $Y$ is connected.
  \enp

\bl{lt1} Let $\mathbf{ G}$ be a connected algebraic group, 
 $\Fr$  a Frobenius map, and let $G=\mathbf{ G}^{\Fr}=\{g\in \mathbf{ G} \mid \Fr(g)=g\}$.
Let $H$ be a subgroup of $G$ and $H_1=xHx^{-1}\leq G$ for some $x\in \mathbf{ G}$. Suppose that $C_{\mathbf{G}}(H)$ is connected. Then $H,H_1$ are conjugate in G.\el

\bp Let $h\in H$. Then $xhx^{-1}\in G$, so $xhx^{-1}=\Fr(x)h\Fr(x)^{-1}$. Then $\Fr(x)^{-1}x\in C_{\mathbf{ G}}(h)$ for all $h\in H$.
Therefore, $\Fr(x)^{-1}x\in C_\GC(H)$. By Lang's theorem, there exists $c\in C_\GC(H)$
such that  $\Fr(x)^{-1}x=\Fr(c)^{-1}c$, so $xc\up\in G$ and $H_1=(xc\up)H(xc\up)^{-1}$.  \enp

\bl{fr1}
Let $H_1$ be as in Lemma {\rm \ref{l33}}. Then $H_1$ is G-conjugate to a subgroup of $M_2\cong G_2(q)$.\el

\bp By Lemma \ref{l33}, $xH_1x\up< M_2 < G$ for some $x\in \mathbf{G}$, so Lemma \ref{lt1}  yields the result.\enp

\begin{remar}
{\em Lemma \ref{fr1} justifies the claim in \cite[p. 2520, line 6]{Z4} stated therein with no proof.
Thus, this fixes a gap in the proof of  \cite[Lemma 4.14]{Z4} which gives a proof of Theorem \ref{td43} for $\ell=0$ and $p>2$.}
\end{remar}

\begin{propo}\label{D66}
Let  $g\in G\cong{}^3D_4(q)$  be a semisimple $p$-element. Suppose that Sylow p-subgroups of G are not cyclic. Let $\theta \in\Irr_F G$ with $\dim\theta  >1$. Then  $\deg\theta  (g)=|g|$.\end{propo}

\begin{proof}
(i) The case with $q=2$  can be settled by a computer computation. (Note that when $p\neq \ell$
one can also use the Brauer character table, and the $p=\ell>3$ case follows from \cite{Z2}; so if $q=2$ then it suffices to deal with the case $p=\ell=3$.) Let $q>2$.

If $g$ is contained in a subgroup isomorphic to $G_2(q)$ then the lemma follows  from Theorem \ref{g2o}.
Suppose the opposite. Then $p$ divides the $r'$-part of $|G|/|G_2(q)|=q^6(q^8+q^4+1)$, i.e. $p|(q^8+q^4+1)$.
 As we assume the Sylow $p$-subgroups of $G$ are non-cyclic, we have  $p|(q^4+q^2+1)$, so
$p$ divides $q^2+\ep q+1$ for some $\ep\in\{\pm 1\}$. In particular, $p>2$.

\smallskip
(ii) Here we consider the case $p > 3$. By \cite{Kl},   $G$ contains a subgroup $H$ isomorphic to $X \circ  Y$,  where
$X\cong \SL^\ep_3(q)$   and $Y\cong C_\ep$, a cyclic subgroup of order $q^2+\ep q+1$. 
Then
$$|G|/|H|=\gcd(3,q-\ep)q^9(q^3+\ep)(q^8+q^4+1)/(q^2+\ep q+1),$$ 
whence $(p,|G|/|H|)=1$. 
Therefore, $H$ contains a \syl of $G$. By Lemma \ref{fr1}, 
$X$ is contained in a subgroup $D \cong G_2(q)$.


Express    $g=xy$ for $x\in  X$  and  $y\in  Y$. Then $y\neq 1$ as otherwise $g \in X< D\cong G_2(q)$.   In addition, by \cite[Lemma 4.13]{Z5}, we may assume that $|x|=|y|=|g|$.

Suppose that $\deg\theta(g)<|g|$. Let $\tau$ be any composition factor of $\theta|_H$ such that $\tau=\phi\otimes \lam$ with $1_X\neq \phi\in\Irr _F X$ and $\lam\in\Irr _F Y$.
Then $\deg\phi(x)<|g| = |x|$. This implies that $\phi$ is a Weil \rep of $X$ and $|g|=|C_\ep|$, see  Lemma \ref{ss3}
if $\ep=1$ and \cite[Proposition 6.1]{TZ8} if $\ep=-1$.
Therefore,  every non-trivial composition factor of $\theta|_X$ is a Weil \rep of $X$.
If  $\ep=-1 $ and $(3,q)=1$ then this  contradicts  Lemma \ref{ti1}. If $3 \nmid(q-\ep)$ and $\ell|(q^4+q^2+1)$ then this  contradicts   Lemma  \ref{r8r}.
Thus, if $\ep=-1$, then we may assume $3|q$ and $\ell \nmid (q^4+q^2+1)$, whence $p \neq \ell$.
If $\ep=1$, then, as $p > 3$ and $q^2+q+1$ is a $p$-power, we have $3 \nmid (q-1)$, and so we may assume that
$\ell \nmid (q^4+q^2+1)$, whence $p \neq \ell$.
In both the cases, we have $Z(X)=1$ and $X \circ Y = X \times Y$.

We can write  $\tau  |_{(X\times Y)}=(\chi_1\otimes \lam_1)
+\ldots   +   (\chi_s\otimes \lam_s)$ as a sum of Brauer characters, where  $\chi_i\in
\Irr_F  X$  and $\lam_i\in  \Irr_F  Y$  for $i=1,\ldots ,s.$  We order
the summands so that $ \chi_i(1)>1$ for $i=1,\ldots ,t\leq s$ and
$\chi_i(1)=1$ for $i>t$.  Let $\phi_i$ be the representation afforded by $\chi_i$, and let $\lam_i$ also denote the respective representation of $Y$ as
this character is linear.  
Then $\deg\phi_i(x)<|g|=|x|$. As $ Z(X)=1$ and $\ell \neq p|(q^2+q\ep+1)$,
every non-trivial $p$-element is regular in $X$, so,  by Lemma \ref{cs2}, either
$\phi$ is liftable or constant on the non-identity $p$-elements of $X$.
By Lemma \ref{ss3} and \cite[Proposition 6.1]{TZ8},
if $1 \leq i \leq t$ then  $\deg\phi_i(x)=|x|-1=\dim \phi_i(1)$, $\phi_i$ is real, and hence $\deg \theta  (g)=|g|-1$. Then $1\notin\spec
\theta  (g)$. (Indeed, otherwise some other $|g|$-root $\nu$ of unity is not in the spectrum of $\theta(g)$; as $p>2$ and $g$ is conjugate to $g^{-1}$ (see for instance \cite[Theorem 1(vi)]{TZ5}), $\nu^{-1}\notin\spec
\theta  (g)$, and hence  $\deg \theta  (g)\leq |g|-2$, which is a contradiction).
Since $\phi_i$ is real, we also have that $1 \notin \spec \phi_i(x)$, and hence
$\spec \phi_i(x)$ consists of all nontrivial $|g|$-roots of unity when $1 \leq i \leq t$. It follows that $\lam_i(y)=1$ for $1 \leq i \leq t$.
As $y \notin \ker(\tau)$ and $1 \notin \spec \theta(g)$, we have $s > t$, and none of $\lam_{t+1}(y)\ld \lam_s(y)$ equals 1. \itf $1$ is not an \ei of
$\theta(x^ky^l)$ whenever $k,l$ are coprime to $p$, whereas $1\in\spec\theta(x)$. However,
by \cite[Lemma 4.13]{Z4}, $x$ is conjugate to some $x^ky^l$ with $|x^k|=|y^l|=|x|$
(the key point used in \cite[Lemma 4.13]{Z4} is that $N_G(T_\ep)$ acts primitively on a maximal torus
$T_\ep \cong C_{q^2+\ep q+1}^2$). This is a contradiction.

\smallskip
(iii) Let $p=3$. Then $|g|$ does not divide $q^2-1$. (Indeed, otherwise $|g|$ divides $q-\ep$ for some $\ep\in\{1,-1\}$,
and, by Lemma \ref{cf6}(i), $g$ is contained in a torus of $G$ of order $(q-\ep)^2$.
By Lemma \ref{sylow}(iii),  the tori of order $(q-\ep)^2$ are conjugate in $G$. A subgroup isomorphic to $G_2$ contains a torus of this order, so it contains a conjugate of  $g$.)

It is known (see claim (*) in the proof of \cite[Proposition 4.8, p. 2517]{Z4})
that $g$ is contained in a subgroup $H\cong X\circ Y$, where $X\cong \SL_2(q)$ and $Y\cong \SL_2(q^3). $
Express $g=xy$ with $x\in X$, $y\in Y$ to be 3-elements.
Let $\tau$ be an \ir constituent of $\theta|_H$. Then $\tau=\phi\otimes \eta$, where $\phi,\eta$ are \ir \reps of $X,Y$, respectively. Hence $\tau(g)=\phi(x)\otimes \eta(y)$. Suppose 
that $\deg\theta(g)<|g|$. Then $\deg\tau(g)<|g|$ implies $\deg\phi(x)<|g|$ and  $\deg\eta(y)<|g|$.  If $|y|\leq |x|$ then $|g|$ divides $q^2-1$; by the above this is not the case. So $|y|> |x|$, and hence $|g|=|y|$.
In this case, choose $\tau$ so that $\dim \eta>1$. Then $\deg \eta(y) < |g|=|y|$ implies by \cite[Theorem 1.1]{DZ1} that $3|(q+1)$. By  \cite[Lemma 3.3]{TZ8} applied to $Y$, it follows that  $q^3+1$ is a $3$-power
if $q$ is even and $(q^3+1)/2$ is a 3-power if $q$ is odd. The former case is ruled out by Lemma \ref{zgm} (as $q>2$). In the latter case $(q^3+1)/2=(q^2-q+1)(q+1)/2$ is a $3$-power implies $q^2-q+1$ and $(q+1)/2$ to be 3-powers, which is false as  $\gcd(q^2-q+1,(q+1)/2)\in\{1,3\}$.
\end{proof}

\bp[Proof of Theorem {\rm \ref{td43}}] The result follows from Lemma \ref{ub1} (and the discussion at the beginning of
the section) if  $p|(q^4-q^2+1)$, and from Proposition \ref{D66} if $p \nmid (q^4-q^2+1)$.
\enp

\section{The case of ${}^2F_4(q)$}


In this section we prove Theorem \ref{mth1} for $G={}^2F_4(q)$, $q=2^{2k+1}$.

\bl{f22}
Theorem {\rm \ref{mth1}} is true for $G= {}^2F_4(2)'$ or ${}^2F_4(2)$.\el

\bp
If and $\ell\neq p$ then the result  follows by inspection of the Brauer character table of $G$ \cite{JLPW}. If $p=\ell=13$ then this follows from \cite{Z2}. So we are left with $p=3,5$.

Note that $G$ has no element of order 9 or 25 and all elements of order 3 and of order 5 are conjugate. In addition, $G$ contains  a subgroup $H\cong {\rm PSL}_2(25)$, so we can assume that $g\in H$.
If $|g|=3$ then the result follows from \cite[Lemma 3.3]{TZ8}.

Let $|g|=5$. Then we can assume that $g$ is contained in  a subgroup $K\cong {\rm PSL}_2(9)$. If $1_K\neq \tau\in\Irr_5(K)$ and $\deg\tau(g)<5$ then $\dim\tau=4$ (\cite[Lemma 3.3]{TZ8}). Therefore, if the lemma is false then $\phi|_K=a\tau+b\cdot 1_G$.
Let $x\in K$ be of order 3. Then $\tau(x)=2$, so   $\phi(x)=2a+b$, where $a\geq 1$.
This implies $\phi(1)\leq 2\phi(x)$, which contradicts the data in the Brauer character table of $G$ for $\ell=5$.    \enp

 So in what follows we assume $q>2$. Observe that if Sylow $p$-subgroups of $G={}^2F_4(q)$ are cyclic, then $p$ divides $q^2-q+1$ or  $q^4-q^2+1$.

\begin{lemma}\label{G2F4}
Theorem {\rm \ref{mth1}} is true if is true if Sylow p-subgroups of G are not cyclic. \end{lemma}

\begin{proof}   Suppose the contrary, that is, $\deg\theta   (g)<|g|$.
As a \syl is not cyclic, one observes that $p|(q^4-1)$. If $p|(q-1)$ or $p|(q^2+1)$ then $g$ is contained in a direct product of two copies of ${}^2B_2(q)$ (see \cite[Table 1]{Z4}). In this case
 the result follows from that for ${}^2B_2(q)$. 

 Suppose that $p|(q+1)$. If $p=3$ then
 $g$ is contained in a subgroup  $H\cong \SU (3,q)$ \cite[Table 1]{Z4}. Then \cite[Lemma 6.1]{TZ8} implies $q=8$, $|g|=9$ and $g$ is contained in a maximal torus of $H$ of order $(q+1)^2$. This is also true if $p>3$. This torus is also a maximal torus of   a subgroup $H_1 \cong \Sp_4(q)$. So we can assume $g\in H_1$, and the result follows from Lemma \ref{p4a}.\end{proof}

\begin{lemma}\label{g2r1}
Theorem {\rm \ref{mth1}} is true if  $p|(q^2-q+1)$.\end{lemma}

\bp  Suppose the contrary. As $q^2-q+1$ divides the order of a subgroup $H\cong \SU_3(q)$, it follows that
$H$ contains a \syl of $G$. Therefore, we assume that $g\in H$.
By \cite[Prop. 6.1(ii)]{TZ8}, $3|(q+1)$, which is impossible since $q=2^{2m+1}$.\enp

The above analysis reduces the proof of Theorem {\rm \ref{mth1}}
 to the case where $p|(q^4-q  ^2+1)$; in particular, the Sylow $p$-subgroups are cyclic.
In view of \cite{Z2} and \cite{Z4}, we may assume that $\ell\neq p$ and that $\phi$ is not liftable.
Note that
 $$q^4-q  ^2+1=(q^2+q\sqrt{2q}+ q+\sqrt{2q}+1)(q^2-q\sqrt{2q}+q-\sqrt{2q}+1),$$
 in fact, $G$ contains maximal tori $T_1,T_2$ of these orders so we may assume that $g\in T_1$ or $ g\in T_2$. The rest of the section is therefore devoted to the proof of

\begin{propo}\label{f44}
Theorem {\rm \ref{mth1}} is true if $p|(q^4-q^2+1)$
unless possibly when $\ell=3$, $q=8$, $p=109$ and   $\phi=\phi_{21}$ in notation of {\rm \cite{Hi2}}.
\end{propo}

First we recall

\bl{h22}
Let $d_\ell$ be the minimum degree of a nontrivial $\ell$-modular \irr of $G={}^2F_4(q)$ with $q>2$.

\begin{enumerate}[\rm(i)]
\item \cite[Theorem 6.1]{Hi2} If $\ell>3$ then   $d_\ell\geq (q-1)(q+1)^2(q^2-q+1)\sqrt{q/2}$.

\item \cite[Theorem 1.4]{Ti6} If $\ell=3$ then  $d_3\geq (q -1)(q^4 + q^3 +q)\sqrt{q/2}$.
\end{enumerate}
\el

Lemma \ref{tr1} implies the following:

\bl{h11}
Let $\phi$ be a non-trivial unipotent Brauer character of G, and $T\in\{T_1,T_2\}$.
Let $T'$ be the subgroup of $\ell'$-elements of $T$ and $1\neq t\in T'$.  If $\phi(t)\geq 0$ or $a=\phi(t)<0$ and $-a(|T|-1)<\phi(1)$ then $\phi|_T$ contains every $\nu\in\Irr T$. In particular, this holds if $q>3$, $a<0$ and $-a<q(q^3+q^2+1)/8$.\el

\begin{proof} We have $|T|-1\leq q^2+q+(q+1)\sqrt{2q}$, so
 $$-a<\frac{(q -1)(q^4 + q^3 +q)\sqrt{q/2}}{(q+1)(q+\sqrt{2q})}= \frac{(q -1)(q^4 + q^3 +q)}{(q+1)2+\sqrt{2q})}$$
and
 $$\frac{q-1}{(q+1)(2+\sqrt{2q})}>\frac{7}{9(2+\sqrt{16})}=\frac{7}{54}>1/8,$$
as $q\geq 8$. \end{proof}

Recall that we have to deal only with the cases where $\ell\neq p$ and $\phi$ is not  liftable. Therefore,
Lemma \ref{cs2} together with Lemma \ref{h11} reduces the proof to the case where $\phi$ is unipotent. Our strategy in proving Proposition \ref{f44}  is to show, using the $\ell$-decomposition numbers of $G$,  that either $\phi(t)\geq 0$ or $\phi(t)<q(q^3+q^2+1)/8$, and then use Lemma \ref{h11}.

\med
Hiss \cite{H11} has determined the decomposition numbers of $G={}^2F_4(q)$ modulo $\ell|(q^4-q^2+1)$ and $\ell|(q^2-q+1)$;  Himstedt \cite{Hi} has computed these for remaining $\ell\neq 2$. Note that Himstedt's tables involve some indetermined values, which leads to certain difficulties below, in particular, for $q=8$.

\med The degrees of the unipotent characters are available in Malle \cite{Ma}. In his notation these are $\chi_k$ with  $k=1\ld 21$. Note that $G={}^2F_4(q)$ has exactly two  maximal tori $T_1,T_2$ (up to conjugation) that satisfies the assumption of Lemma \ref{sy1}. These are of the aforementioned orders $|T_1|=q^2+q\sqrt{2q}+q+\sqrt{2q}+1$ and $|T_2|=  q^2-q\sqrt{2q}+q-\sqrt{2q}+1$. Let $\eta$ be as in
Lemma \ref{sy1}. Then $\eta$ can be computed by taking the congruences of $\chi_k(1)$ modulo $|T_i|$ for $i=1,2$. The result
is recorded in Hiss \cite[p. 886 and p. 884]{Hi}, and we display it in Table 3 below. Note that fourth column
lists the characters of $p_i$-defect 0 for $i=1,2$.

\bigskip
\begin{center}{Table 3: Unipotent character values at $1\neq t\in(T_1\cup T_2)$}

{\small
 \vspace{10pt}
\noindent\begin{tabular}{|l|c|c|c|c| }
\hline
  &  $\chi_{i}(t)=1$   & $\chi_{i}(t)=-1$ &$\chi_i(t)=0$   \\
\hline
$1\neq t\in T_1$& $i=1,4,5,6,13,15,16,19,20$ &  $i=7,8,10$   & $i=2,3,9,11,12,14,17,18,21 $
   \\
\hline
$1\neq t\in T_2$& $ i=1,4,7,8,12,17,18,19,20$ & $i=5,6,9$   &  $i=2,3,10,11,13,14,15,16,21$ \\
\hline
\end{tabular}}\end{center}

\bigskip

In what follows, we will use data of Table 3 without further referring. Also, denote
$$\phi_1=q-1,~\phi_2=q+1,~\phi_8'=q+\sqrt{2q}+1,~\phi_8''=q-\sqrt{2q}+1,~\phi_{8}=q^2+1,~\phi_{12}=q^2-q+1,$$
$$\phi'_{24}=q^2+q\sqrt{2q}+q+\sqrt{2q}+1=|T_2|,~~\phi''_{24}=q^2-q\sqrt{2q}+q-\sqrt{2q}+1=|T_1|,$$
and $\phi_{24}=q^4-q^2+1,$ so that $\phi_{24}=\phi'_{24}\phi''_{24}=|T_1|\cdot|T_2|$.

The notation of  \ir characters is as in \cite{Ma}. Note that their parametrization is the same as in Hiss \cite{H11} (where $\xi_j$ is used for $\chi_j$).

\subsection{Brauer characters}
Note that we only need to deal with non-liftable \ir Brauer characters as $\ell\neq p$. The set of such characters is denoted by $\Irr^0_\ell(G). $ Note that tori $T_1,T_2$ satisfy the assumption of Lemma \ref{sy1}. By Lemma  \ref{cs2}, every character in $\Irr^0_\ell(G) $ is unipotent and constant on the $\ell'$-elements of $T_i$, $i=1,2$. The result is trivial if the constant in question equals 0 (for $i\in\{1,2\}$), as in this case $\phi|_{T_i}$ is a multiple of the regular character $\rho_{T_i}^{\reg}$.
So below we only consider the cases where $\phi(t)\neq 0$ for $1\neq t\in T_i$. If $\ell$ divides $|T_i|)$ then we write $T_i'$ for the subgroup of $\ell'$-elements of $T_i$, $i=1,2$.

Every \ir Brauer character agrees on $\ell'$-elements with some  integral linear combination of
ordinary characters.  If an $\ell$-modular character $\tau$, say, agrees on $\ell'$-elements with
$\sum a_i\chi_i$, where the $\chi_i$'s are ordinary characters, we simply write $\tau=\sum a_i\chi_i$. (or $\tau=_{\ell'}\sum a_i\chi_i$.)

\bl{A-F} Let $\phi\in\Irr^0_\ell(G)$, $T\in\{T_1,T_2\}$, and $t \in T$. Then either $\phi(t)\geq 0$,
or $-\phi(t)<q(q^3+q^2+1)/8$, or $\ell=3,$ $q=8$ and $T=T_1$.
\el

\begin{proof}
(i) We start with primes $\ell$ for which Sylow $\ell$-subgroups are cyclic. This means that $\ell$ divides either $q^2-q+1$ or $|T_1|$ or $T_2$. These are the cases (a), (b), (c) below.

\med
(a) \st $\ell|(q^2-q+1)$. Then, by Hiss \cite[Theorem 4.5]{H11}, $|\Irr^0_\ell(G)|=2 $, and for $\phi_1, \phi_2\in \Irr^0_\ell(G)$ we have
$\phi_1=\chi_{11}-1_G$, $\phi_2=\chi_{4}-\chi_{11}-\chi_{19}-\chi_{20}+1_G$.
Then $\phi_1(t)=-1$ and $\phi_2(t)= 0$ for every $1\neq t\in T_1\cup T_2$.
So the result follows from Lemma \ref{h11}.

\med
(b) Let $\ell$ divide $|T_1|$. Then, by Hiss \cite[Theorem 4.7]{H11}, $|\Irr^0_\ell(G)|=4 $, and we have $\phi_1=\chi_{10}-1_G$, $\phi_{2}=\chi_{7}-\chi_5-\chi_{15}-\chi_{19}$, $\phi_3=\chi_8-\chi_{6}-\chi_{16}-\chi_{20}$ and $\phi_4=\chi_4 -\phi_1-\phi_2-\phi_3$. Then for every $1\neq t\in T_2$ we have $\phi_{1}(t)=-1$, $\phi_2(t)= \phi_3(t)=1$ and $\phi_4(t)=0$.

In turn, for every $1\neq t\in T_1'$ we have $\phi_{1}(t)= -2$, $\phi_2(t)= \phi_3(t)= -4$,
and $\phi_4(t)= 11$. So the result follows from Lemma \ref{h11}.

\med
(c) Let $\ell$ divide $|T_2|$. Then, by Hiss \cite[Theorem 4.6]{H11}, $|\Irr^0_\ell(G)|=4 $, and we have
$$\phi_{1}=\chi_{9}-1_G,\phi_2=\chi_7-\chi_5,\phi_3=\chi_8-\chi_{6},\phi_4=\chi_{4}-
\chi_{9}+1_G.$$
Then for every $1\neq t\in T_1$ we have $\phi_{1}(t)=-1$, $\phi_2(t)= \phi_3(t)=-2$, $\phi_4(t)=2$, and
for every $1\neq t\in T_2'$ we have $\phi_{1}(t)= -2$,
$\phi_2(t)=\phi_3(t)= 2$, and $\phi_4(t) =3$. 
So the result follows from Lemma \ref{h11}.

\med
(ii) Next we consider the cases where Sylow $\ell$-subgroups are not cyclic. Our main reference is  \cite{Hi2}.
Our ordering of the unipotent characters is as in \cite{H11} and \cite{Ma}, which is different from those in
\cite{Hi2}. So we indicate by arrows the correspondence of our characters to those in   \cite{Hi2}:
$$\begin{aligned}1_G=\chi_1\ra \chi_1,~\chi_2\ra \chi_4,~\chi_3\ra \chi_{18},~\chi_4\ra \chi_{21},~\chi_5\ra \chi_2,~\chi_6\ra \chi_3,~\chi_7\ra \chi_{19},\\
\chi_8\ra \chi_{20},~\chi_9\ra \chi_5,~\chi_{10}\ra \chi_6,~\chi_{11}\ra \chi_7,~\chi_{12}\ra \chi_8,~\chi_{13}\ra \chi_9,~\chi_{14}\ra \chi_{10},\\
\chi_{15}\ra \chi_{11},~\chi_{16}\ra \chi_{12},~\chi_{17}\ra \chi_{13},~\chi_{18}\ra \chi_{14},~
\chi_{19}\ra \chi_{15},~\chi_{20}\ra \chi_{16},~\chi_{21}\ra \chi_{17}.\end{aligned}$$

\med
(d) $\ell|(q^2-1)$. Here $|\Irr^0_\ell(G)|=0 $.

\med
$(e)$ $3\neq \ell|(q^2+1) $. Then $\Irr^0_\ell(G)=\{\phi_i: i=4,7,18,21 \}$ in notation of \cite{Hi2}. We have

$\phi_4=\chi_2-1_G$ so $\phi_4(t)=-1$ and   for all $1\neq t\in T_1\cup T_2$.

$\phi_7=\chi_{11}-\chi_2$ and $\phi_{18}=\chi_{3}-a\chi_{21}-\chi_{11}+\chi_2$, where $2\leq a\leq(q^2-2)/3, $ see \cite[Theorem 3.2]{Hi2}. Then $\phi_7(t)=\phi_{18}(t)=0$ for all $1\neq t\in T_1\cup T_2$.

Furthermore, we have
$$\phi_{21}=\chi_{4}-e\phi_{18}-d\chi_{21}-c\chi_{13}-b\chi_{12}-\phi_7-\phi_4,$$
where $1\leq b\leq (q+\sqrt{2q})/4$, $1\leq c\leq (q-\sqrt{2q})/4$, $2\leq d\leq (q^2+2)/3$, $2\leq e\leq (q+2)/2$,
 so $\phi_{21}(t)=2-c$ for all  $1\neq t\in T_1$ and $\phi_{21}(t)=2-b$ for all  $1\neq t\in T_2$.
 In both the cases $\phi_{21}(t)<q(q^3+q^2+1)/8$, so the result follows by Lemma \ref{h11}.

\med
(f) $ \ell=3 $. Then $(3,q^4-q^2+1)=1$ as $3|(q^2-1)$. In this case there are 22 unipotent Brauer characters and $\Irr^0_\ell(G) =\{\phi_{5,1}, \phi_i:i\in \{ 4,7,8,10,15,18,21\}$.

In this case the expressions of $\phi_i$ in terms of ordinary characters $\chi_j$ depend of parameters
which are not determined in full but satisfy certain inequalities. This are
$$2\leq a\le q, 0\leq b\le (q+\sqrt{2q})/4, 0\leq c\le (q-\sqrt{2q})/4,
2\leq d\le q^2,
1\leq e\le (q+2)/2, $$ $$
0\leq x_7\le q/2, 0\leq x_8\le(q+3\sqrt{2q}+4)/12, 0\leq x_{10}\le (q-2)/6, 1\leq x_{15}\le (q+1)/3,$$ $$
0\leq x_{18}\le q(q-1), 1\leq x_{21}\le q^3.$$

We have

 $\phi_{4}=\chi_{2}-1_G$, so $\phi_{4}(t)=-1$ for $1\neq t\in T_1\cup T_2;$
\med

$\phi_{5,1}=\chi_{20}-\chi_{21}$, whence $\phi_{5,1}(t)=1$ 
for $1\neq t\in T_1\cup T_2$.

\med
 $\phi_{8}=\chi_{12}-x_8\phi_{5,1}$. 
 So $\phi_{8}(t) 
 =1-x_8$ if $T=T_1$ and $-x_8$ if $T=T_2$. As $x_8<q(q^3+q^2+1)/8$, Lemma \ref{h11} applies.

\med
 $\phi_{7}=\chi_{11}-1_G-\phi_4-x_7\phi_{5,1}=\chi_{11}-\chi_{2}-x_7\phi_{5,1}$.
Note that $\chi_{11}(t) =0$ for $1\neq t\in T_1\cup T_2$.  So $\phi_{7}(t)=-x_7$, and
$x_7 <q(q^3+q^2+1)/8$.

\med
 $\phi_{10}=\chi_{14}-\phi_8-x_{10}\phi_{5,1}$.
Note that $\chi_{14}(t)=0$ for $1\neq t\in T_1\cup T_2$.
So $\phi_{10}(t)=-x_8-x_{10}$ if $t\in T_2$ and $1-x_8-x_{10}$ if $t\in T_1$.
Here $x_8+x_{10}<q+(q-2)/6<q(q^3+q^2+1)/8.$

\med
$\phi_{15}=\chi_{19}-x_{15}\phi_{5,1}$, 
where $1\leq x_{15}\le (q+1)/3$. As $\chi_{19}(t)=1$ for $1\neq t\in T_1\cup T_2$, we have
$\phi_{15}=1-x_{15}$. As above, $x_{15}<q(q^3+q^2+q)/8$ yields the result.


$\phi_{18}=\chi_{3}-\phi_7-x_{18}\phi_{5,1}-a\phi_{15}$. 
 So $\phi_{18}(t)=x_7-x_{18} +a(x_{15}-1)$
for $1\neq t\in T_1\cup T_2$. As $x_{18}+a<q(q^3+q^2+1)/8$, the result follows from Lemma \ref{h11}.

$\phi_{21}=\chi_4-\phi_{4}-\phi_7-x_{21}\phi_{5,1}-b\phi_8-c\phi_{10}-d\phi_{15}-e\phi_{18}$. 
So  $\phi_{21}(t)=2+x_7+d(x_{15}-1)-x_{21}-e(x_7-x_{18} +a(x_{15}-1)) -b\phi_8(t)-c\phi_{10}(t)=
2+(1-e)x_7+(d-ae)x_{15}-x_{21}+ex_{18}-a-d -b\phi_8(t)-c\phi_{10}(t)$.
In addition, $\phi_8(t)=1-x_8$ if $T=T_1$ and $-x_8$ if $T=T_2$, and $\phi_{10}(t)=1-x_8-x_{10}$ if $1\neq t\in T_1$
and $-(x_8+x_{10})$ if $t\in T_2$.  So $-b\phi_8(t)-c\phi_{10}(t)=-b(1-x_8)-c(1-x_8-x_{10})=(b+c)(x_8-1)+cx_{10}$ if $1\neq t\in T_1$,
and $-b(-x_8)-c(-x_8-x_{10})=(b+c)x_8+cx_{10}$ if $1\neq t\in T_2$. So
$$\phi_{21}(t)= \begin{cases}2-a-d+(1-e)x_7+(d-ae)x_{15}-x_{21}+ex_{18}+(b+c)(x_8-1)+cx_{10}& 1\neq t\in T_1,\\
2-a-d+(1-e)x_7+(d-ae)x_{15}-x_{21}+ex_{18}+(b+c)x_8+cx_{10}& 1\neq t\in T_2.
\end{cases}$$
Therefore
$$-\phi_{21}(t)\leq a+d+(e-1)x_7+(ae-d)x_{15}+x_{21}\leq   q+q^2+\frac{q^2}{4} +\frac{q(q-2)(q+3)}{6} +q^3<2q^3.$$
If $q>8$ then $2q^3<q(q^3+q^2+q)/8$, so the result follows by Lemma \ref{h11}.

Let $q=8$.    If $t\in T_2$ then
$$-\phi_{21}(t)\cdot (|T|-1)\leq 688 \cdot 36=24768
  <q(q-1)(q^3+q^2+1)\sqrt{q/2}=112 \cdot 577=64624,$$
so Lemma \ref{h11} yields the result.

 As $\phi_{21}$ is constant on $T_1\smallsetminus \{1\}$,
  all non-trivial $|t|$-roots of unity are \ei of $\phi_{21}(t)$. \end{proof}

\begin{remar}\label{for9.4}
{\em Observe that $|T_1|=109$, $|T_2|=37$. By \cite[Corollary 4.2]{Hi2}, $c=1$ for $q=8$. Then $-\phi_{21}(t)\leq -2+a+d+(e-1)x_7+(ae-d)x_{15}+x_{21}\leq -2+q+q^2+\frac{q^2}{2}+\frac{q^2+2q+4)(q+1)}{6}+q^3=6+64+32+ 111+512=725$ and $-\phi_{21}(t)(|T_1|-1)\leq 78300$. The lower bound for $\dim \phi$ suggested in
 \cite{Ti6}, see Lemma \ref{h22}, is $q(q-1)(q^3+q^2+1)\sqrt{q/2}$ for $q=8$ yields $\dim\phi\geq 64624$.
Note that if $1_{T_1}$ is not a constituent of $\phi|_{T_1}$ then $\dim\phi$ is a multiple of $|T_1|-1=108$.
As 64624 is not a multiple of 108, we conclude that $\dim\phi\geq 64692$.
 So if $t\in T_1$ then the question remains open.}
\end{remar}

\bp[Proof of Theorem {\rm \ref{mth1}}] The result follows from Lemma \ref{sz1} when $G={}^2B_2(q)$,
 Lemma \ref{re1} when $G={}^2G_2(q)$, Theorem \ref{g2o} and Lemma \ref{gg2} when $G=G_2(q)$, Theorem  \ref{td43} when $G={}^3D_4(q)$, and from Lemmas \ref{G2F4}, \ref{g2r1}, Proposition \ref{f44}, and Remark \ref{for9.4} when $G={}^2F_4(q)$.
\enp

\end{document}